\documentclass[10pt]{amsart}
\usepackage{amssymb}
\usepackage{epsfig}
\usepackage{graphicx}
\numberwithin{equation}{section}
\usepackage{cite}
\usepackage{anysize}
\marginsize{3cm}{3cm}{3cm}{3cm}

\input xy
\xyoption{all}  

\usepackage[usenames,dvipsnames]{color}
\theoremstyle{plain}
\newtheorem{prop}{Proposition}[section]

\newtheorem{theo}[prop]{Theorem}

\newtheorem{lemm}[prop]{Lemma}

\theoremstyle{definition}

\numberwithin{equation}{section}

\newcommand{\bA}{\mathbb A}
\newcommand{\bC}{\mathbb C}

\newcommand{\bP}{\mathbb P}
\newcommand{\bQ}{\mathbb Q}
\newcommand{\bZ}{\mathbb Z}
\newcommand{\bR}{\mathbb R}
\newcommand{\bF}{\mathbb F}
\newcommand{\R}{\mathbb R}
\newcommand{\Z}{\mathbb Z}

\newcommand{\cO}{\mathcal O}

\newcommand{\sH}{\mathsf H}

\newcommand{\sZ}{\mathsf Z}

\newcommand{\Pic}{{\rm Pic}}
\newcommand{\Val}{{\rm Val}}

\newcommand{\PGL}{{\rm PGL}}

\newcommand{\mathbs}{{\mathbf s}}

\newcommand{\ch}{\mathrm{ch}\,}

\begin{document}
\title[Equivariant compactifications of $\mathrm{PGL}_2$]
{Distribution of Rational points of bounded height on equivariant compactifications of $\mathrm{PGL}_2$ I}

\author{Ramin Takloo-Bighash}
\address{Department of Mathematics\\
UIC, 851 S. Morgan Str\\
Chicago, IL 60607\\
USA}
\email{rtakloo@math.uic.edu}
\author{Sho Tanimoto}
\address{Department of Mathematical Sciences\\
University of Copenhagen\\
Universitetspark 5\\
2100 Copenhagen $\emptyset$\\
Denmark}
\email{sho@math.ku.dk}

\date{\today}

\begin{abstract}
We study the distribution of rational points of bounded height on a one-sided equivariant compactification of $\PGL_2$ using automorphic representation theory of $\PGL_2$.
\end{abstract}

\maketitle

%\tableofcontents

\section{Introduction}

A driving problem in Diophantine geometry is to find asymptotic formulae for the number of rational points on a projective variety $X$ with respect to {\it a height function}. In \cite{batyrev-manin}, Batyrev and Manin formulated a conjecture relating the generic distribution of rational points of bounded height to certain geometric invariants on the underlying varieties. This conjecture has stimulated several research directions and has lead to the development of tools in analytic number theory, spectral theory, and ergodic theory. Although the strongest form of Manin's conjecture is known to be false (e.g., \cite{bt-exam}, \cite{BL13}, and \cite{LR15}), there are no counterexamples of Manin's conjecture in the class of {\it equivariant compactifications of homogeneous spaces} whose stabilizers are connected subgroups.

There are mainly two approaches to the study of the distribution of rational points on equivariant compactifications of homogeneous spaces. One is the method of ergodic theory and mixing, (e.g., \cite{GMO08}, \cite{GO11}, and \cite{GTBT11}), which posits that the counting function of rational points of bounded height should be approximated by the volume function of height balls. This method has been successfully applied to prove Manin's conjecture for wonderful compactifications of semisimple groups. The other approach is the method of height zeta functions and spectral theory (e.g., \cite{STBT} and \cite{PGL2} as well as \cite{BT98}, \cite{BT-general}, \cite{VecIII}, and \cite{ST15}); this method solves cases of toric varieties, equivariant compactifications of vector groups, biequivariant compactifications of unipotent groups, and wonderful compactifications of semisimple groups.

In all of these results, one works with a compactification $X$ of a group $G$ that are bi-equivariant, i.e.  the right and left action of $G$ on itself by multiplication extends to $X$. The study of one-sided equivariant compactifications remains largely open, and the only result in this direction is \cite{TT12} which treats the case of the $ax+b$-group under some technical conditions.

From a geometric point of view, one-sided equivariant compactifications of reductive groups are different from bi-equivariant compactifications, and their birational geometry is more complicated. For example, in the case of bi-equivariant compactifications of reductive groups, the cone of effective divisors is generated by boundary components. In particular, when reductive groups have no character, the cone of effective divisors is a simplicial cone. However, this feature is absent for one-sided equivariant compactifications of reductive groups, and one can have more complicated cones for these classes of varieties. This has a serious impact on the analysis of rational points.

In all previous cases where spectral theory or ergodic theory are applied, the main term of the asymptotic formula for the counting function associated to the anticanonical class arises from the trivial representation component of the spectral expansion of the height zeta function, assuming  that the group considered has no character. The trivial representation component has been studied by Chambert-Loir and Tschinkel in \cite{volume}. They showed that when the cone of effective divisors is generated by boundary components, the trivial representation component coincides with Manin's prediction. However, if the cone of effective divisors is not generated by boundary components, then the trivial representation component does not suffice to account for the main term of the height zeta function. 

In this paper, we study a one-sided equivariant compactification of $\PGL_2$ whose cone of effective divisors is not generated by boundary components. We use the height zeta functions method and automorphic representation theory of $\PGL_2$. A new feature is that for the height function associated with the anti-canonical class, the main pole of the zeta function comes not from the trivial representation, but from constant terms of Eisenstein series; indeed, the contribution of the trivial representation is cancelled by a certain residue of Eisenstein series, c.f. \S \ref{sect:eisenstein} for details. In particular, it would appear as though ergodic theory methods cannot shed light on this situation, as, at least in principle, these methods only study the contribution of one-dimensional representations.

Let us express our main result in qualitative terms:
\begin{theo}\label{mainthm}
Let $X$ be the blow up of $\bP^3$ along a line defined over $\bQ$.
The variety $X$ satisfies Manin's conjecture, with Peyre's constant, for any big line bundle over $\bQ$. 
\end{theo}

The blow up of $\bP^3$ along a line is a toric variety and an equivariant compactification of a vector group and thus our result is covered by previous works on Manin's conjecture. However, our proof is new in the sense that we explicitly used the structure of one-sided equivariant compactifications of $\PGL_2$, but not biequivariant, and we develop a method using automorphic representation theory of $\PGL_2$, which can be applied to more general examples of equivariant compactifications of $\PGL_2$, some of which are not covered by previous works. 

Let us describe our result more precisely.  Consider the following equivariant compactification of $G=\PGL_2$:
$$
\PGL_2 \ni \begin{pmatrix} a&b\\c&d \end{pmatrix} \mapsto (a:b:c:d) \in \bP^3.
$$
The boundary divisor $D$ is a quadric surface defined by $ad - bc  = 0$. We consider the line $l$ on $D$ defined by
$$
\begin{pmatrix} 0&1 \end{pmatrix}\begin{pmatrix} a&b\\c&d \end{pmatrix} = 0.
$$
We let $X$ be the blowup of $\bP^3$ along the line $l$. 
Note that the left action of $\PGL_2$ on $\bP^3$ acts transitively on lines in the ruling of $l$,
so geometrically their blow ups are isomorphic to each other.
The smooth projective threefold $X$ is an equivariant compactification of $\PGL_2$ over $\bQ$ and the natural right action on $\PGL_2$ extends to $X$. 
The variety $X$ extends to a smooth projective scheme over $\mathrm{Spec}(\bZ)$ and the action of $\PGL_2$ also naturally extends to this integral model.
We let $U$ be the open set consisting of the image of $\PGL_2$ in $X$. 

We denote the strict transformation of $D$ by $\tilde{D}$, and the exceptional divisor by $E$. The boundary divisors $\tilde{D}$ and $E$ generate $\Pic (X)_{\bQ}$, however, the boundary divisors do not generate the cone of effective divisors $\Lambda_{\rm{eff}}(X)$. The cone of effective divisors is generated by $E$ and $P = \frac{1}{2}\tilde{D} - \frac{1}{2}E$ which corresponds to the projection to $\bP^1$.

Let $F$ be a number field. For an archimedean place $v\in \rm{Val}(F)$, the height functions are defined by
$$
\mathsf H_{E, v}(a,b,c,d) = \frac{\sqrt{|a|_v^2 + |b|_v^2 + |c|_v^2 + |d|_v^2}}{\sqrt{|c|_v^2 + |d|_v^2}},
$$
$$
\mathsf H_{\tilde{D}, v}(a,b,c,d) = \frac{\sqrt{|a|_v^2 + |b|_v^2 + |c|_v^2 + |d|_v^2}\sqrt{|c|_v^2 + |d|_v^2}}{|ad-bc|_v},
$$
For a non-archimedean place $v \in \rm{Val}(F)$, we have
$$
\mathsf H_{E, v}(a,b,c,d) = \frac{\max \{ |a|_v, |b|_v, |c|_v, |d|_v\}}{\max \{|c|_v, |d|_v\}},
$$
$$
\mathsf H_{\tilde{D}, v}(a,b,c,d) = \frac{\max \{ |a|_v, |b|_v, |c|_v, |d|_v\}\max \{|c|_v, |d|_v\}}{|ad-bc|_v},
$$

Thus the local height pairing is given by
$$
\mathsf H_v(g, (s, w) ) = \mathsf H_{\tilde{D},v}(g)^{s} \mathsf H_{E,v}(g)^{w} 
$$
For ease of reference, we let 
$$
\mathsf H_1(g) = \max \{|c|_v, |d|_v\}
$$
and 
$$
\mathsf H_2(g) = \max \{ |a|_v, |b|_v, |c|_v, |d|_v\}. 
$$
The complexified height function: 
$$
\mathsf H(g, \mathbs)^{-1} = \mathsf H_1(g)^{w-s} \mathsf H_2(g)^{-s-w} |\det g|^{s}. 
$$

\

The global height paring is given by
$$
\mathsf H(g, \mathbf s) = \prod_{v\in \rm{Val}(F)} \mathsf H_v(g, \mathbf s) : G(\bA_F)\times \Pic(X)_{\bC} \rightarrow \bC^{\times}
$$
The anti-canonical class $-K_X$ is equal to $2 \tilde{D} + E$, and as such we have 
$$
\mathsf H_{-K_X}(g) =  \mathsf H_{\tilde{D},v}(g)^{2} \mathsf H_{E,v}(g). 
$$
A more precise version of the theorem is the following: 
\begin{theo}\label{mainthm2}
Let $C$ be a real number defined by 
$$
\zeta(3) C = 5 \gamma - 3 \log 2 + \frac{3}{4} \log \pi - \log \Gamma\left(\frac{1}{4}\right) - \frac{24}{\pi^2} \zeta'(2) - \frac{\zeta'(3)}{\zeta(3)} - 4. 
$$
Then there is an $\eta >0$ such that as $B \to \infty$, 
$$
\# \{ \gamma \in U(\bQ) \mid \mathsf H_{-K_X}(\gamma) < B \} = \frac{1}{\zeta(3)} B (\log B) + C B + O(B^{1-\eta}).
$$
\end{theo}

We will show in \S \ref{sect:peyre} that this is indeed compatible with the conjecture of Peyre \cite{peyre95}. 

\

Our method is based on the spectral analysis of the height zeta function given by 
$$
\mathsf Z(s, w) = \sum_{\gamma \in G(F) } \mathsf H(\gamma, s, w)^{-1}.
$$
Namely, for $g \in \PGL_2(\bA)$, we let 
$$
\mathsf Z(g; s, w) = \sum_{\gamma \in G(F) } \mathsf H(\gamma g; s, w)^{-1}.
$$
For $\Re s, \Re w$ large, $\mathsf Z(.; s, w)$ is in $L^2(\PGL_2(\bQ) \backslash \PGL_2(\bA))$ and is continuous on $\PGL_2(\bA)$. We will then use the spectral theory of automorphic functions to analytically continue $\mathsf Z$ to a large domain.  The main result will be a corollary of the following general statement: 

\begin{theo}\label{mainthm3} Our height zeta function has the following decomposition:
$$
\sZ(s, w)
=   
\frac{\Lambda(s+w-2)}{\Lambda(s+ w)} E(s-3/2, e) 
-   \frac{\Lambda(s+w-1)}{\Lambda(s+ w)} E(s-1/2, e)
 + \Phi(s, w)
$$
with $\Phi(s, w)$ a function holomorphic for $\Re s > 2 - \epsilon$ and $\Re(s+ w) > 2$ for some $\epsilon >0$.
Here $\Lambda$ is the completed Riemann zeta function defined in \S \ref{zeta}, and $E(s, g)$ is the Eisenstein series defined in \S \ref{subsect:eisenstein}. 
\end{theo}

\

As in previous works, the proof of the theorem is based on the the spectral decomposition theorem proved in \cite{PGL2} and some approximation of Ramanujan conjecture \cite{Sarnak}. But a new idea is needed here. Recall that the method of \cite{PGL2} is based on the analysis of matrix coefficients--what facilitates this is the fact that the height functions considered there are bi-$K$-invariant. Namely, we need to find bounds for integrals of the form 
\begin{equation}\label{height-cusp}
\int_{G(\bA)} \sH(g, s)^{-1} \phi(g) \, \mathrm dg 
\end{equation}
with $\phi(g)$ a cusp form.  If $\sH$ is left-$K$-invariant, then we may write the integral as 
$$
\frac{1}{{\text{ vol }}K} \int_{G(\bA)} \sH(g, s)^{-1} \int_K\phi(kg) \, \mathrm dk \, \mathrm dg. 
$$
The function $g \mapsto \int_K\phi(kg) \, \mathrm dk$ is roughly a linear combination of products of local spherical functions coming from various local components of automorphic representations.  Approximations to the Ramanujan conjecture give us bounds for spherical functions, and this in turn gives rise to appropriate bounds for our integrals. 

\

As mentioned above, the height functions we consider here are not bi-$K$-invariant.   We use representation theoretic versions of Whittaker functions, which are the adelic analogues of Fourier expansions of holomorphic modular forms.   The idea is to write 
$$
\phi(g) = \sum_{\alpha \in \bQ^\times} W_\phi\left( \begin{pmatrix} \alpha \\ & 1 \end{pmatrix} g \right)
$$
with 
$$
W_\phi(g) = \int_{\bQ \backslash \bA} \phi\left(\begin{pmatrix} 1 & x \\ & 1 \end{pmatrix} g \right) \psi^{-1}(x) \, \mathrm dx 
$$
with $\psi: \bA \to \bC^\times$ the standard non-trivial additive character. 
Using these Whittaker functions we can write the integral \eqref{height-cusp} as the infinite sum 
$$
\sum_{\alpha \in \bQ^\times} \int_{G(\bA)} \sH(g, s)^{-1}W_\phi\left( \begin{pmatrix} \alpha \\ & 1 \end{pmatrix} g \right) \, \mathrm dg 
$$
Whittaker functions are Euler products, and there are explicit formulae for the local components of these functions expressing their values in terms of the Satake parameters of local representations. Again, approximations to the Ramanujan conjecture are used to estimate these local functions. 

Even though for the sake of clarity we state our results over $\bQ$, everything we do goes through with little or no change for an arbitrary number field $F$. The local computations of \S \ref{sect:one-dimensional} and \S\ref{sect:cuspidal} and the global results of \S\ref{sect:cuspidal} and \S\ref{sect:eisenstein} remain valid.  The spectral decomposition of \S\ref{sect:spectral} needs to be adjusted in the following way. We have 
$$
\sZ(\mathbs, g)_{\mathrm{res}} = \frac{1}{\mathrm{vol} \, (G(F)\backslash G(\bA))}  \sum_\chi  \langle \sZ(\mathbs, \cdot), \chi \circ \det \rangle \chi(\det(g))
$$
where the (finite) sum is over all unramified Hecke characters $\chi$ such that $\chi^2 =1$.  When the class number of the field $F$ is odd, e.g. when $F = \bQ$, the sum consists of a single term corresponding to $\chi=1$. At any rate, Lemma \ref{lemm: integral_computation} shows that the only term that may contribute to the main pole is $\chi=1$. A computation as in the case of $\bQ$ shows that the term coming from the trivial representation is cancelled, and we arrive at Theorem \ref{mainthm3}. Theorem \ref{mainthm} then immediately follows. Except for the determination of the value of the constant $C$, Theorem \ref{mainthm2} is valid as well. At this point we do not know how to compute the constant $C$ in general as there is no Kronecker Limit Formula available for an arbitrary number field. 

We expect that our method has more applications to Manin's conjecture, and we plan to pursue these applications in a sequel \cite{TBT15}. A limitation of our method is that it can only be applied to the general linear group, as automorphic forms on other reductive groups typically do not possess Whittaker models, e.g. automorphic representations on symplectic groups of rank larger than one corresponding to holomorphic Siegel modular forms \cite{Howe-PS}.

\

This paper is organized as follows.  \S\ref{sect:preliminaries} contains some background information. The proof of the main theorem has four basic steps: Step 1, the analysis of the one dimensional representations presented in \S\ref{sect:one-dimensional}; Step 2, the analysis of cuspidal representations and Step 3, the analysis of Eisenstein series, presented, respectively, in \S\ref{sect:cuspidal}
and \S\ref{sect:eisenstein}; and finally Step 4, the spectral theory contained in \S\ref{sect:spectral} where we put the results of the previous sections together to prove the main theorem of the paper. In \S\ref{sect:peyre} we show that our results are compatible with the conjecture of Peyre. 

\

\noindent
{\bf Acknowledgements.}
We wish to thank Daniel Loughran, Morten Risager, Yiannis Sakellaridis, Anders S\"odergren, and Yuri Tschinkel for useful communications. 
We would also like to thank referees for careful reading which significantly improves the exposition of our paper. The first author's work on this project was partially supported by the National Security Agency and the Simons Foundation.  The second author is supported by Lars Hesselholt's Niels Bohr Professorship.

\section{Preliminaries}\label{sect:preliminaries}
We assume that the reader is familiar with the basics of the theory of automorphic forms for $\PGL_2$ at the level of \cite{Gelbart} or \cite{Godement}. For ease of reference we recall here some facts and set up some notation that we will be using in the proof of the main theorem. 

\subsection{Riemann zeta}\label{zeta}  As usual $\zeta(s)$ is the Riemann zeta function, and $\Lambda(s)$ the completed zeta function defined by 
$$
\Lambda(s) = \pi^{-s/2} \Gamma(s/2) \zeta(s). 
$$
The function $\Lambda(s)$ has functional equation
$$
\Lambda(s) = \Lambda(1-s), 
$$
and the function 
$$
\Lambda(s) + \frac{1}{s} - \frac{1}{s-1}
$$
has an analytic continuation to an entire function.

\subsection{An integration formula} 
We will need an integration formula. If $H$ is a unimodular locally compact group, and $S$ and $T$ are two closed subgroups, 
such that $ST$ covers $H$ except for a set of measure zero, and $S \cap T$ is compact, then 
$$
\mathrm dx = \mathrm d_l s \, \mathrm d_r t
$$
is a  Haar measure on $H$ where $\mathrm d_ls$ is a left invariant haar measure on $S$,
and $\mathrm d_r t$ is a right invariant haar measure on $T$. In particular, if $T$ is unimodular, then 
$$
\mathrm dx = \mathrm d_l s \, \mathrm dt
$$
is a Haar measure. We will apply this to the Iwasawa decomposition.

As we defined in the introduction, let $G = \PGL_2$. Suppose that $F$ is a number field.
For each place $v \in \Val(F)$, we denote its completion by $F_v$.
Then we have the Iwasawa decomposition:
\[
 G(F_v) = P(F_v)K_v
\]
where $P$ is the standard Borel subgroup of $G$ i.e., the closed subgroup of upper triangler matrices,
and $K_v$ is a maximal compact subgroup in $G(F_v)$.
(When $v$ is a non-archimedean place, $K_v = G(\cO_v)$ where $\cO_v$ is the ring of integers in $F_v$.
When $v$ is a real place, $K_v = \mathrm{SO}_2(\bR)$.)
It follows from the integration formula that for any measurable function $f$ on $G(F_v)$, we have
$$
\int_{G(F_v)} f(g) \, \mathrm dg = \int_{F_v} \int_{F_v^\times} \int_{K_v} 
f\left( \begin{pmatrix} 1 & x \\ & 1 \end{pmatrix} \begin{pmatrix} a \\ & 1 \end{pmatrix} k\right) |a|^{-1} \, \mathrm dk \, \mathrm da^\times \, \mathrm d x. 
$$
If $v$ is a non-archimedean place and $f$ is a function on $\PGL_2(F_v)$ which is invariant on the right under $K_v$ then we have
$$
\int_{G(F_v)} f(g) \, \mathrm dg = 
\sum_{m\in \bZ} q^{m} \int_{F_v} f\left( \begin{pmatrix} 1 & x \\ & 1 \end{pmatrix} \begin{pmatrix} \varpi^m \\ & 1 \end{pmatrix}\right) \, \mathrm dx,
$$
where $\varpi$ is an uniformizer of $F_v$.
If $v$ is an archimedean place, then instead we have
$$
\int_{G(F_v)} f(g) \, \mathrm dg = 
\int_{F_v} \int_{F_v^\times} 
f\left( \begin{pmatrix} 1 & x \\ & 1 \end{pmatrix} \begin{pmatrix} a \\ & 1 \end{pmatrix}\right) |a|^{-1}  \, \mathrm da^\times \, \mathrm d x.
$$
We will use these integration formulae often without comment. 

\subsection{Whittaker models}
Let $N$ be the unipotent radical of the standard Borel subgroup in $\PGL_2$, i.e., the closed subgroup of upper triangler matrices. 
For a non-archimedean place $v$ of $\bQ$, we let $\theta_v$ be a non-trivial character of $N(\bQ_v)$. Define $C_{\theta_v}^\infty (\PGL_2(\bQ_v))$ 
to be the space of smooth complex valued functions on $\PGL_2(\bQ_v)$ satisfying 
$$
W(ng) = \theta_v(n) W(g)
$$
for all $n \in N(\bQ_v), g \in \PGL_2(\bQ_v)$. For any irreducible admissible representation $\pi$ of $\PGL_2(\bQ_v)$ the intertwining space 
$$
{\rm Hom}_{\PGL_2(\bQ_v)}(\pi, C_{\theta_v}^\infty (\PGL_2(\bQ_v))) 
$$
is at most one dimensional; if the dimension is one, we say $\pi$ is generic, and we call the corresponding realization of $\pi$ as a space of $N$-quasi-invariant functions the Whittaker model of $\pi$. 

\

We recall some facts from \cite{Godement}, \S 16. Let $\pi$ be an unramified principal series representation 
$\pi = \mathrm{Ind}_P^G(\chi \otimes \chi^{-1})$, with $\chi$ unramified, where $P$ is the standard Borel subgroup of $G$. 
Then $\pi$ has a unique $K_v$ fixed vector. The image of this $K_v$-fixed vector in the Whittaker model, call it $W_\pi$, 
will be $K_v$-invariant on the right, and $N$-quasi-invariant on the left. 
By Iwasawa decomposition in order to calculate the values of $W_\pi$ it suffices to know the values of the function along the diagonal subgroup. We have 
\begin{align*}
W_\pi\begin{pmatrix} \varpi^m \\ & 1 \end{pmatrix} & = \begin{cases} q^{-m/2} \sum_{k=0}^m \chi(\varpi^k) \chi^{-1}(\varpi^{m-k}) & m \geq 0 ; \\ 
0 & m < 0
\end{cases}\\
& = \begin{cases} q^{-m/2} \frac{\chi(\varpi)^{m+1}- \chi(\varpi)^{-m-1}}{\chi(\varpi) - \chi(\varpi)^{-1}} & m \geq 0 ; \\ 
0 & m < 0.
\end{cases}
\end{align*}
Written compactly we have 
$$
W_\pi \begin{pmatrix} a \\ & 1 \end{pmatrix} = |a|^{1/2} \ch_\cO(a) 
\frac{\chi(\varpi) \chi(a) - \chi(\varpi)^{-1} \chi(a)^{-1}}{\chi(\varpi) - \chi(\varpi)^{-1}},
$$
where
\[
\ch_\cO(a) =
 \begin{cases}
  1 & \text{ if $a \in \cO$}\\
  0 & \text{ if $a \not\in \cO$}.
 \end{cases}
\]

Also by definition 
$$
W_\pi\left(\begin{pmatrix} 1 & x \\ & 1 \end{pmatrix} g k\right) = \psi_v(x) W_\pi(g),
$$
where $\psi_v : \bQ_v \rightarrow \mathbb S^1$ is the standard additive character of $\bQ_v$.

\begin{lemm}
We have 
$$
\sum_{m=0}^\infty q^{m(1/2 - s)}W_\pi\begin{pmatrix} \varpi^m \\ & 1 \end{pmatrix} =  L(s, \pi) 
$$
where 
$$
L(s, \pi) := \frac{1}{(1- \chi(\varpi) q^{-s})(1- \chi^{-1}(\varpi) q^{-s})}. 
$$
\end{lemm} 

\

Let us also recall the automorphic Fourier expansion \cite[P. 85]{Gelbart}. If $\phi$ is a cusp form on $\PGL_2$ we have 
$$
\phi(g) = \sum_{\alpha \in \bQ^\times} W_\phi\left(\begin{pmatrix} \alpha \\ & 1 \end{pmatrix} g\right) 
$$
with 
$$
W_\phi(g) = \int_{\bQ \backslash \bA} \phi \left( \begin{pmatrix} 1 & x \\ & 1 \end{pmatrix} g \right) \psi^{-1}(x) \, \mathrm dx. 
$$

\subsection{Eisenstein series}\label{subsect:eisenstein}
By Iwasawa decomposition, any element of $\PGL_2(\bQ_v)$ can be written as 
$$
g_v = n_v a_v k_v
$$
with $n_v \in N(\bQ_v), a_v \in A(\bQ_v), k_v \in K_v$. Define a function $\chi_{v, P}$ by 
$$
\chi_{v, P}: g_v = n_v a_v k_v \mapsto |a_v|_v
$$
where we have represented an element in $A(\bQ_v)$ in the form 
$$
\begin{pmatrix} a_v \\ & 1 \end{pmatrix}. 
$$
We set 
$$
\chi_P : = \prod_v \chi_{v, P}. 
$$
We note that for $\gamma \in P(\bQ)$, we have $\chi_P(\gamma g) = \chi_P(g)$ for any $g \in G(\bA)$. 
Moreover, $\chi_P^{-1}$ is the usual height on $\mathbb P^1(\bQ) = P(\bQ) \backslash G(\bQ)$ which is used in the study of height zeta functions for generalized flag varieties in \cite{fmt}. Define the Eisenstein series $E(s, g)$ by 
$$
E(s, g) = \sum_{\gamma \in P(\bQ)\backslash G(\bQ)} \chi(s, g) 
$$
where $\chi(s, g) := \chi_P(g)^{s+ 1/2}$.  For later reference we note the Fourier expansion of the Eisenstein series \cite[Equation 3.10]{Gelbart} in the following form: 
$$
E(s, g) = \chi_P(g)^{s+ 1/2} + \frac{\Lambda(2s)}{\Lambda(2s+1)} \chi_P(g)^{-s + 1/2} + \frac{1}{\zeta(2s+1)}\sum_{\alpha \in \bQ^\times} W_s\left(\begin{pmatrix} \alpha \\ & 1 \end{pmatrix} g \right). 
$$
Here $W_s((g_v)_v) = \prod_v W_{s, v}(g_v)$, where for $v < \infty$, $W_{s, v}$ is the normalized $K_v$-invariant Whittaker function for the induced representation $\mathrm{Ind}_P^G(|.|^s \otimes |.|^{-s})$, and for $v=\infty$, 
$$
W_{s, v}(g) = \int_{\bR} \chi_{P,v}\left(w \begin{pmatrix} 1 & x \\ & 1 \end{pmatrix} g\right)^{s+ 1/2} e^{2 \pi i x} \, \mathrm dx, 
$$
where $ w = \begin{pmatrix} & 1 \\ -1 \end{pmatrix}$ is a representative for the longest element of the Weyl group. 
This integral converges when $\Re(s)$ is sufficiently large, and has an analytic continuation to an entire function of $s$. 
We also have the functional equation 
$$
E(s, g) = \frac{\Lambda(2s)}{\Lambda(1+ 2s)}E(-s, g),
$$
where $g \in G(\bA)$. We note that 
$$
{\rm Res}_{s=1/2} E(s, g) = \frac{1}{2 \Lambda(2)} = \frac{3}{\pi}. 
$$

\

The following lemma, generalized by Langlands \cite{Langlands}, is well-known: 

\begin{lemm} 
We have 
$$
 {\rm Res}_{y = 1/2} E(y, e)= \frac{1}{\mathrm{vol} \, (G(\bQ)\backslash G(\bA))},
$$
where $e \in G(\bA)$ is the identity.
\end{lemm}
\begin{proof}
For any smooth compactly supported function $f$ on $(0, +\infty)$, we define the Mellin transform by 
$$
\hat{f}(s) = \int_{0}^\infty f(x) x^{-s} \, \frac{\mathrm dx}{x}. 
$$
Mellin inversion says for $\sigma \gg 0$ 
$$
f(x) = \frac{1}{2 \pi i} \int_{\sigma - i \infty}^{\sigma + i \infty} \hat{f}(s) x^s \, \mathrm ds.
$$
We also define for $g \in G(\bA)$ 
$$
\theta_f(g) = \sum_{\gamma \in P(\bQ) \backslash G(\bQ)} \chi_P(\gamma g)^{1/2} f(\chi_P(\gamma g)). 
$$
We have 
\begin{align*}
\theta_f(g) & =  \sum_{\gamma \in P(\bQ) \backslash G(\bQ)} 
\frac{\chi_P(\gamma g)^{1/2}}{2 \pi i}  \int_{\sigma - i \infty}^{\sigma + i \infty} \hat{f}(s) \chi_P(\gamma g)^s \, \mathrm ds \\
& = \frac{1}{2 \pi i} \int_{\sigma - i \infty}^{\sigma + i \infty} \hat{f}(s) E(s, g) \, \mathrm ds \\
& = \hat{f}\left(\frac{1}{2}\right) {\rm Res}_{s = 1/2} E(s, g) + \frac{1}{2 \pi i} \int_{- i \infty}^{+ i \infty} \hat{f}(s) E(s, g) \, \mathrm ds. 
\end{align*}
As the residue of the Eisenstein series at $s=1/2$ does not depend on $g$, we have 
\begin{align*}
\int_{G(\bQ)\backslash G(\bA)} \theta_f(g) \, \mathrm dg & = \hat{f}\left(\frac{1}{2}\right) 
{\rm Res}_{s = 1/2} E(s, e) \mathrm{vol}(G(\bQ) \backslash G(\bA)) \\
& + \frac{1}{2 \pi i} \int_{G(\bQ) \backslash G(\bA)} \int_{- i \infty}^{+ i \infty} \hat{f}(s) E(s, g) \, \mathrm ds \, \mathrm dg.
\end{align*}
We now calculate the integral of $\theta_f$ a different way. We have 
\begin{align*}
\int_{G(\bQ)\backslash G(\bA)} \theta_f(g) \, \mathrm dg & = \int_{P(\bQ) \backslash G(\bA)} \chi_P(g)^{1/2} f(\chi_P(g)) \, \mathrm dg  \\
& = \int_K \int_{A(\bQ) \backslash A(\bA)} \int_{N(\bQ) \backslash N(\bA)} \chi_P(nak)^{1/2} f(\chi_P(nak)) \, \mathrm dn \, \mathrm da \, \mathrm dk \\
& = \int_{\bQ^\times \backslash \bA^\times} |a|^{-1/2}f(|a|)\, \mathrm d^\times a \\
& = \mathrm{vol}(\bQ^\times \backslash \bA^c) \int_0^\infty x^{-1/2} f(x)\, \frac{\mathrm dx}{x} \\
& = \mathrm{vol}(\bQ^\times \backslash \bA^c) \hat{f}\left(\frac{1}{2}\right),
\end{align*}
where $\bA^c$ is the kernel of the norm $N: \bA^\times \rightarrow \bR^\times$.
Since $\bQ$ has class number one, we conclude that $ \mathrm{vol}(\bQ^\times \backslash \bA^c)=1$. Comparing the two expressions for $\int \theta_f$ gives the lemma. 
\end{proof} 

\subsection{Spectral expansion} 
Let $f$ be a smooth bounded right $K$- and left $\PGL_2(\bQ)$-invariant function on $\PGL_2(\bA)$ all of whose derivatives are also smooth and bounded. Here we recall a theorem from \cite{PGL2} regarding the spectral decomposition of such a function. 

\

We start by fixing a basis of right $K$-fixed functions for $L^2(G(\bQ) \backslash G(\bA))$. We write 
$$
L^2(G(\bQ) \backslash G(\bA))^K = L_{\mathrm{res}}^K \oplus L_{\mathrm{cusp}}^K \oplus L_{\mathrm{eis}}^K,
$$
where $L_{\mathrm{res}}^K$ the trivial representation part, $L_{\mathrm{cusp}}^K$ the cuspidal part, and $L_{\mathrm{eis}}^K$ the Eisenstein series part.
%Since $\bQ$ has class number one, $L_{\mathrm{res}}^K$ is the trivial representation. 
An orthonormal basis of this space is the constant function 
$$
\phi_{\mathrm{res}}(g) = \frac{1}{\sqrt{\mathrm{vol} \, (G(\bQ)\backslash G(\bA))}}. 
$$
The projection of $f$ onto $L_{\mathrm{res}}^K$ is given by 
$$
f(g)_{\mathrm{res}} = \langle f, \phi_{\mathrm{res}} \rangle \phi_{\mathrm{res}}(g) 
= \frac{1}{\mathrm{vol} \, (G(\bQ)\backslash G(\bA))} \int_{\PGL_2(\bQ)\backslash \PGL_2(\bA)} f(g) \, \mathrm dg. 
$$

\

Next, we take an orthonormal basis $\{ \phi_\pi \}_\pi$ for $L_{\mathrm{cusp}}^K$ where $\pi$ runs over all automorphic cuspidal representations with a $K$-fixed vector. We have 
$$
f(g)_{\mathrm{cusp}} = \sum_\pi \langle f, \phi_\pi \rangle \phi_\pi(g). 
$$
This is possible because $\mathrm{dim}(\pi^K) = 1$ for any $\pi$.
Indeed, by Tensor product theorem, we have $\pi \cong \bigotimes' \pi_v$ where $\pi_v$ is a local cuspidal representation.
Taking the $K$-invariant part, we conclude that $\pi^K \cong \bigotimes' \pi_v^{K_v}$.
Since $\pi$ has a non-zero $K$-fixed vector, the local representation $\pi_v$ also has a non-zero $K_v$-fixed vector.
This implies that $\pi_v$ is the induced representation $\mathrm{Ind}_P^G(\chi\otimes \chi^{-1})$ of some unramified character $\chi$ for $P(F_v)$.
Now it follows from the Iwasawa decomposition that $\dim (\pi_v^{K_v})=1$.

\

Finally we consider the projection onto the continuous spectrum. We have 
$$
f(g)_{\mathrm{eis}} = \frac{1}{4 \pi} \int_{\bR} \langle f , E(it, .) \rangle E(it, g) \, \mathrm dt.  
$$

\

We then have 
$$
f(g) = f(g)_{\mathrm{res}} + f(g)_{\mathrm{cusp}}+ f(g)_{\mathrm{eis}}
$$
as an identity of continuous functions. 

\section{Step one: one dimensional automorphic characters}\label{sect:one-dimensional}
In this step we study the function $\sZ(s, g)_{\mathrm{res}}$.  We have 
$$
\int_{G(\bQ)\backslash G(\bA)} \sZ(g, \mathbs) \, \mathrm dg =\int_{G(\bA)} \sH(g, \mathbs)^{-1} \,\mathrm dg 
= \prod_v \int_{G(\bQ_v)}  \sH_v(g, \mathbs)^{-1} \, \mathrm dg
$$

We consider the local integral
$$ 
  \int_{G(\bQ_v)} \sH_v(g, \mathbs)^{-1} \, \mathrm dg
 =  \int_{G(\bQ_v)} \sH_1(g)^{w-s} \sH_2(g)^{-s-w} |\det g|^{s} \, \mathrm dg. 
$$

\subsection{Non-archimedean computation} 

Here we will use Chambert-Loir and Tschinkel's formula of height integrals.
Let $F$ be a number field.
We fix the standard integral model $\mathcal G$ of $\PGL_2$ over $\cO_F$ where $\cO_F$ is the ring of integers for $F$.
Let $\omega$ be a top degree invariant form on $\mathcal G$ defined over $\cO_F$ which is a generator for $\Omega^3_{\mathcal G/\mathrm{Spec}(\cO_F)}$. 
For any non-archimedean place $v \in \mathrm{Val}(F)$, we denote its the $v$-adic completion by $F_v$, the ring of integers by $\cO_v$, the residue field by $\bF_v$. 
We write the cardinality of $\bF_v$ by $q_v$. For any uniformizer $\varpi$,we have $|\varpi|_v = q_v^{-1}$. Then it is a well-known formula of Weil that

$$
\int_{G(\cO_v)} \mathrm d |\omega|_v = q_v^{-3} \# G(\bF_v) = 1-\frac{1}{q_v^2}.
$$
We denote this number by $a_v$. Then the normalized Haar measure is
$$
\mathrm d g_v = \frac{\mathrm d |\omega|_v}{a_v},
$$
so that $\int_{G(\cO_v)} \mathrm d g_v = 1$.
The variety $X$ has a natural integral model over $\rm{Spec}(\cO_F)$, and it has good reduction at any non-archimedean place $v$. 
Thus Chambert-Loir and Tschinkel's formula applies to our case. (See \cite[Proposition 4.1.6]{volume}.) 
Note that $-\mathrm{div} (\omega) = 2 \tilde{D} + E$, so we have
\begin{align*}
\int_{G(F_v)} \sH_v (g_v, s, w)^{-1} \, \mathrm d g_v &= a_v^{-1} \int_{X(F_v)} \sH_{\tilde{D}, v}^{-s}\sH_{E, v}^{-w} \, \mathrm d |\omega|_v\\
&= a_v^{-1} \int_{X(F_v)} \sH_{\tilde{D}, v}^{-(s-2)}\sH_{E, v}^{-(w-1)} \, \mathrm d \tau_{X, v}\\
&= a_v^{-1}(q_v^{-3}\# G(\bF_v) + q_v^{-3}\frac{q_v-1}{q_v^{s-1}-1} \# \tilde{D}^\circ (\bF_v)\\&+ q_v^{-3}\frac{q_v-1}{q_v^{w}-1} \# E^\circ (\bF_v) + q_v^{-3} \frac{q_v-1}{q_v^{s-1}-1}\frac{q_v-1}{q_v^{w}-1} \# \tilde{D} \cap E(\bF_v) )\\
&= \frac{1-q_v^{-(s+w)}}{(1-q_v^{-(s-1)})(1-q_v^{-w})}.
\end{align*}

\subsection{The archimedean computation}  We have 
\begin{align*}
 \int_{\PGL_2(\bR)} & \sH_1(g)^{w-s}\sH_2(g)^{-s-w} \mathopen|\det g \mathclose|^s \, \mathrm dg  \\
 & = \int_\bR \int_{\bR^\times} \sH_1 \left( \begin{pmatrix} 1 & x \\ & 1 \end{pmatrix} \begin{pmatrix} \alpha \\ & 1 \end{pmatrix} \right)^{w-s} 
 \sH_2 \left( \begin{pmatrix} 1 & x \\ & 1 \end{pmatrix} \begin{pmatrix} \alpha \\ & 1 \end{pmatrix} \right)^{-w-s}
 |\alpha|^{s-1} \, \mathrm d^\times \alpha \, \mathrm d x \\
 &= \int_\bR \int_{\bR^\times}(\alpha^2 + x^2 + 1)^{\frac{-s-w}{2}}|\alpha|^{s-1} \, \mathrm d^\times \alpha \, \mathrm d x.
\end{align*}
Do a change of variable $\alpha = \sqrt{x^2 + 1} \beta$ to obtain 
$$
\left(\int_{\bR^\times} (\beta^2 + 1)^{\frac{-s-w}{2}} |\beta|^{s-1} \, \mathrm d^\times \beta\right). \left(\int_\bR (x^2 + 1)^{-w/2-1/2} \, \mathrm dx \right)
$$
Now we invoke a standard integration formula.  Equation 3.251.2 of \cite{Tables} says 
$$
\int_0^\infty x^{\mu-1} (1+x^2)^{\nu -1} \, \mathrm dx = \frac{1}{2} B(\frac{\mu}{2}, 1 - \nu - \frac{\mu}{2})
$$
provided that $\Re \mu >0$ and $\Re (\nu + \frac{1}{2} \mu ) < 1$. This implies that 
$$
\int_{\bR^\times} (\beta^2 + 1)^{\frac{-s-w}{2}} |\beta|^{s-1} \, \mathrm d \beta^\times = B(\frac{s-1}{2}, \frac{w+1}{2}) = \frac{\Gamma(\frac{s-1}{2})\Gamma(\frac{w+1}{2})}{\Gamma(\frac{s+w}{2})}
$$
provided that $\Re (s) > 1$ and $\Re (w) > -1/2$. Similarly, 
$$
\int_\bR (x^2 + 1)^{-w/2-1/2} \, \mathrm dx = B(\frac{1}{2}, \frac{w}{2})=\frac{\Gamma(\frac{1}{2})\Gamma(\frac{w}{2})}{\Gamma(\frac{w+1}{2})} = \sqrt{\pi} \frac{\Gamma(\frac{w}{2})}{\Gamma(\frac{w+1}{2})}
$$
provided that $\Re (w) > 0$. Consequently, our integral is equal to
$$
 \sqrt{\pi} \frac{\Gamma(\frac{s-1}{2})\Gamma(\frac{w}{2})}{\Gamma(\frac{s+w}{2})}
$$
if $\Re(s) > 1$ and $\Re (w) > 0$. 

\

Thus if $F=\bQ$ we obtain
$$
\sZ(s, g)_{\mathrm{res}} = \int_{G(\bA)} \sH(g, s, w)^{-1} \, \mathrm d g = \sqrt{\pi} \frac{\Gamma(\frac{s-1}{2})\Gamma(\frac{w}{2})}{\Gamma(\frac{s+w}{2})} \frac{\zeta(s-1)\zeta(w)}{\zeta(s+w)}= \frac{\Lambda(s-1)\Lambda(w)}{\Lambda(s+w)}
$$
where 
$$
\Lambda(u) = \pi^{-u/2}\Gamma(u/2) \zeta(u) 
$$
is the completed Riemann zeta function.

\subsection{An integral computation} For use in a later section we compute a certain type of $p$-adic integral. 
Suppose we have a function $f$ given by the following expression: 
$$
f\left(n\begin{pmatrix} a \\ & 1 \end{pmatrix}  k\right) = |a|^\tau 
$$
for a fixed complex number $\tau$. Here $n \in N(F)$ and $k \in K$ where $F$ is a local field. We would like to compute the integral 
$$
\int_{G(F)} f(g) \sH_1(g)^{w-s} \sH_2(g)^{-s-w} \mathopen|\det g\mathclose|^s \, \mathrm dg. 
$$
By the integration formula this is equal to 
\begin{align*}
\int_F\int_{F^\times} & |a|^{s+\tau - 1} \max \{ |a|, |x|, 1\}^{-s-w} \, \mathrm d^\times a \, \mathrm dx  \\
& = \int_F\int_{F^\times} |a|^{s+\tau - 1} \max \{ |a|, |x|, 1\}^{-(s+\tau)-(w-\tau)} \, \mathrm d^\times a \, \mathrm dx \\
& = \int_{G(F)}  \sH_1(g)^{(w-\tau)-(s+\tau)} \sH_2(g)^{-(s+\tau)-(w-\tau)} \mathopen|\det g\mathclose|^{s+\tau} \, \mathrm dg, 
\end{align*}
by the integration formula. We state this computation as a lemma: 
\begin{lemm}
\label{lemm: integral_computation}
For $f$ as above we have 
$$
\int_{G(F)} f(g) \sH_1(g)^{w-s} \sH_2(g)^{-s-w} \mathopen|\det g\mathclose|^s \, \mathrm dg = \frac{\zeta_F(s+\tau-1)\zeta_F(w-\tau)}{\zeta_F(s+w)},  
$$
if $F$ is non-archimedean. In the case where $F = \bR$, the value of the integral is 
$$
\sqrt{\pi} \frac{\Gamma(\frac{s+\tau-1}{2})\Gamma(\frac{w-\tau}{2})}{\Gamma(\frac{s+w}{2})}. 
$$
\end{lemm} 

\section{Step two: the cuspidal contribution}\label{sect:cuspidal}

In this section $\pi$ is an automorphic cuspidal representation of $\PGL_2$ with a $K$-fixed vector. We denote by $\pi^K$ the space of $K$-fixed vectors in $\pi$ which is one dimensional, and we let $\phi_\pi$ be an orthonormal basis for $\pi^K$. We let 
$$\sZ(s, g)_{\mathrm{cusp}} = \sum_\pi  \langle \sZ(s, .), \phi_\pi \rangle \phi_\pi(g). 
$$
We have 
$$
\langle \sZ(s, .), \phi_\pi \rangle = \int_{G(\bA)} \phi_\pi(g) \sH(\mathbs, g)^{-1} \mathrm dg. 
$$

By the automorphic Fourier expansion we have 
$$
\langle \sZ(s, \cdot), \phi_\pi \rangle = \sum_{\alpha \in F^\times} \int_{G(\bA)}
W_{\phi_\pi}\left(\begin{pmatrix} \alpha \\ & 1 \end{pmatrix} g\right) \sH(\mathbs, g)^{-1} \, \mathrm dg
$$ 
The good thing about the use of the Whittaker function is that they have Euler products, so we may write: 
$$
W_{\phi_\pi}(g) = \prod_v W_{\pi_v}(g_v),
$$
where $\pi \cong \bigotimes '\pi_v$ is the restricted product of local representations.
For $\alpha \in F_v^\times$ we set 
$$
J_{\pi_v}(\alpha) : =  \int_{G(F_v)} W_{\pi_v}\left(\begin{pmatrix} \alpha \\ & 1 \end{pmatrix}g\right) 
\sH_1(g)^{w-s} \sH_2(g)^{-s-w} \mathopen|\det g\mathclose|^{s} \, \mathrm dg
$$

\subsection{$v$ non-archimedean}
\label{subsec: cusp_non-archimedean}

In this case, $\pi_v$ is an unramificed principal series representation,
so it has the form of 
\[
 \mathrm{Ind}_{P}^{G}(\chi \otimes \chi^{-1}),
\]
where $\chi$ is an unramified character of $P(F_v)$ which only depends on $\pi_v$.

We will need the following straightforward lemma: 

\begin{lemm}
For an unramified quasi-character $\eta$ and $y \in F_v$ define 
$$
I(\eta, y) = \int_{|u|>1} \eta(u) \psi_v(yu) \, \mathrm du,
$$
where $\psi_v : F_v \rightarrow \mathbb S^1$ is the standard additive character.

Then for $y \not\in \cO$, $I(\eta, y)=0$. If $y \in \cO$, then 
$$
I(\eta, y) = \frac{1-\eta(\varpi)^{-1}}{1- q^{-1} \eta(\varpi)}\eta(y)^{-1}|y|^{-1} - \frac{1-q^{-1}}{1- q^{-1}\eta(\varpi)}; 
$$
in particular, if $y \in \cO^\times$, then $I(\eta, y) = - \eta(\varpi)^{-1}$. 
\end{lemm} 

\

We have 
\begin{align*}
J_{\pi_v}(\alpha) = & \sum_{m \in \Z} q^m\cdot q^{-ms} \int_{F_v} 
W_{\pi_v} \left( \begin{pmatrix} \alpha \\ & 1 \end{pmatrix}  \begin{pmatrix} 1 & x \\ & 1 \end{pmatrix} \begin{pmatrix} \varpi^m \\ & 1 \end{pmatrix} \right) 
\sH_1 \left(\begin{pmatrix} 1 & x \\ & 1 \end{pmatrix} \begin{pmatrix} \varpi^m \\ & 1 \end{pmatrix} \right)^{w-s} \\
& \,\,\,\,\,\,\,\,\,\,\,\,\, \sH_2 \left(\begin{pmatrix} 1 & x \\ & 1 \end{pmatrix}\begin{pmatrix} \varpi^m \\ & 1 \end{pmatrix} \right)^{-s-w}\, \mathrm dx \\
=& \sum_{m\in \bZ} q^{m - ms}W_{\pi_v}\begin{pmatrix} \alpha \varpi^m \\ & 1 \end{pmatrix}  
\int_{F_v} \sH_2 \left(\begin{pmatrix} 1 & x \\ & 1 \end{pmatrix} \begin{pmatrix} \varpi^m \\ & 1 \end{pmatrix} \right)^{-s-w} 
\psi_v(\alpha x) \, \mathrm dx \\
=& \sum_{m \in \bZ} q^{m - ms}W_{\pi_v}\begin{pmatrix} \alpha \varpi^m \\ & 1 \end{pmatrix}  
\int_{F_v} \max \{ 1, q^{-m}, |x|\}^{-s-w} \psi_v(\alpha x) \, \mathrm dx
\end{align*}

\

The first observation is that if $\alpha \not\in \cO$, then $J_v(\alpha)=0$. In fact, 
in order for $W_{\pi_v}\begin{pmatrix} \alpha \varpi^m \\ & 1 \end{pmatrix}$ to be non-zero, we need to have $m \geq -\mathrm{ord} \, \alpha >0$. In this case, 
$$
 \max \{ 1, q^{-m}, |x|\} =  \max \{ 1, |x|\}. 
$$
Next, 
$$
\int_{F_v} \max \{ 1, |x|\}^{-s-w} \psi_v(\alpha x) \, \mathrm dx = \int_\cO \psi_v(\alpha x) \, \mathrm dx + \int_{|x|>1} |x|^{-s-w} \psi_v(\alpha x) \, \mathrm dx; 
$$
the first integral is trivially zero, and the second integral is zero by the lemma. 

\

We also calculate $J_{\pi_v}(\alpha)$ for $\alpha \in \cO^\times$ by hand. By what we saw above, 
\begin{align*}
J_{\pi_v}(\alpha) & = \sum_{m \geq 0} q^{m - ms}W_{\pi_v}\begin{pmatrix}  \varpi^m \\ & 1 \end{pmatrix}  \int_{F_v} \max \{ 1, |x|\}^{-s-w} \psi_v(\alpha x) \, \mathrm dx \\
& = \sum_{m \geq 0} q^{m - ms}W_{\pi_v}\begin{pmatrix}  \varpi^m \\ & 1 \end{pmatrix}  
\left( \int_\cO \psi_v(\alpha x) \, \mathrm dx + \int_{|x|>1} |x|^{-s-w} \psi_v(\alpha x) \, \mathrm dx\right)\\
& = (1- q^{-s-w})\sum_{m \geq 0} q^{m - ms}W_{\pi_v}\begin{pmatrix}  \varpi^m \\ & 1 \end{pmatrix}  \,\,\,\,\, \text{ (after using the lemma)} \\
& = (1- q^{-s-w})\sum_{m \geq 0} q^{m (1/2- (s-1/2))}W_{\pi_v}\begin{pmatrix}  \varpi^m \\ & 1 \end{pmatrix} \\
& =(1- q^{-s-w}) L(s-1/2, \pi_v). 
\end{align*}

\

Suppose that $\alpha \in \cO$, i.e., $\alpha =\varpi^k$ where $k \geq 0$.
Using the integration formula, we have
\begin{align*}
J_{\pi_v}(\alpha) &= \int_{G(F_v)} W_{\pi_v}\left(\begin{pmatrix} \alpha \\ & 1 \end{pmatrix}g\right) \sH(g, s, w)^{-1} \, \mathrm dg\\
&=  \sum_{m \in \Z} q^m \int_{F_v} 
W_{\pi_v} \left(\begin{pmatrix} \alpha \\ & 1 \end{pmatrix}  \begin{pmatrix} 1 & x \\ & 1 \end{pmatrix} \begin{pmatrix} \varpi^m \\ & 1 \end{pmatrix} \right) 
\sH\left(\begin{pmatrix} 1 & x \\ & 1 \end{pmatrix}\begin{pmatrix} \varpi^m \\ & 1 \end{pmatrix}, s, w \right)^{-1}\, \mathrm dx\\
&=  \sum_{m \in \Z} q^m \int_{F_v} 
W_{\pi_v} \left(\begin{pmatrix} \alpha \\ & 1 \end{pmatrix}  \begin{pmatrix} 1 & x \\ & 1 \end{pmatrix} \begin{pmatrix} \varpi^m \\ & 1 \end{pmatrix} \right) 
\sH\left(\begin{pmatrix} 1 & x \\ & 1 \end{pmatrix}\begin{pmatrix} \varpi^m \\ & 1 \end{pmatrix}, s, w \right)^{-1}\, \mathrm dx \\
&=  \sum_{m \in \Z} q^m \int_{F_v} 
W_{\pi_v} \left(  \begin{pmatrix} 1 & \alpha x \\ & 1 \end{pmatrix} \begin{pmatrix} \alpha \varpi^m \\ & 1 \end{pmatrix} \right) 
\sH\left(\begin{pmatrix} 1 & x \\ & 1 \end{pmatrix}\begin{pmatrix} \varpi^m \\ & 1 \end{pmatrix}, s, w \right)^{-1}\, \mathrm dx \\
&=  \sum_{m \in \Z} q^m 
W_{\pi_v} \left(  \begin{pmatrix} \alpha \varpi^m \\ & 1 \end{pmatrix} \right) \int_{F_v}  
\sH \left( \begin{pmatrix} 1 & x \\ & 1 \end{pmatrix}\begin{pmatrix} \varpi^m \\ & 1 \end{pmatrix}, s, w \right)^{-1} \psi_v(\alpha x)\, \mathrm dx.
\end{align*}
Using an explicit computation of Whittaker functions, we have
$$
J_{\pi_v}(\alpha) = \sum_{m \geq -k} q^{\frac{m-k}{2}} \frac{\chi(\varpi)^{m+k+1} - \chi(\varpi)^{-m-k-1}}{\chi(\varpi)-\chi(\varpi)^{-1}} 
\int_{F_v}  \sH\left(\begin{pmatrix} 1 & x \\ & 1 \end{pmatrix}\begin{pmatrix} \varpi^m \\ & 1 \end{pmatrix}, s, w \right)^{-1} \psi_v(\alpha x)\, \mathrm dx.
$$
Then we decompose this infinite sum into two parts:
\begin{align*}
&\sum_{m \geq -k} q^{\frac{m-k}{2}}\chi(\varpi)^{m+k+1} \int_{F_v}  
\sH\left(\begin{pmatrix} 1 & x \\ & 1 \end{pmatrix}\begin{pmatrix} \varpi^m \\ & 1 \end{pmatrix}, s, w \right)^{-1} \psi_v(\alpha x)\, \mathrm dx\\
&= |\alpha|_v^{\frac{1}{2}}\chi(\varpi)\int_{F_v^{\times}} \int_{F_v} |t|^{-\frac{1}{2}} \chi(\alpha t) \ch_{\cO}(\alpha a) 
\sH\left(\begin{pmatrix} 1 & x \\ & 1 \end{pmatrix}\begin{pmatrix} t \\ & 1 \end{pmatrix}, s, w \right)^{-1} \psi_v(\alpha x)\, \mathrm dt^{\times} \mathrm dx\\
&= |\alpha|_v^{\frac{1}{2}}\chi(\varpi) \int_{P(F_v)}  \sH(p, s, w )^{-1} \psi_v(\alpha x) 
|t|^{-\frac{1}{2}} \chi(\alpha t) \ch_{\cO}(\alpha t)\, \mathrm dp,
\end{align*}
where $P$ is a Borel subgroup and $\mathrm dp$ is a right invariant Haar measure.
Similar computation works for the second part, so we have
\begin{align*}
J_{\pi_v}(\alpha) &=  \frac{|\alpha|_v^{\frac{1}{2}}}{\chi(\varpi)-\chi(\varpi)^{-1}}
(\chi(\varpi) \int_{P(F_v)}  \sH(p, s, w )^{-1} \psi_v(\alpha x) |t|^{-\frac{1}{2}} \chi(\alpha t) \ch_{\cO}(\alpha t)\, \mathrm dp\\
& \,\,\,\,\,\,\,\,\,\,\,\,\,\,\, -\chi(\varpi)^{-1}\int_{P(F_v)}  \sH(p, s, w )^{-1} \psi_v(\alpha x) |t|^{-\frac{1}{2}}
\chi(\alpha t)^{-1} \ch_{\cO}(\alpha t)\, \mathrm dp \left. \right)\\
&= \frac{|\alpha|_v^{\frac{1}{2}}}{\chi(\varpi)-\chi(\varpi)^{-1}} (\chi(\varpi) J_{\pi_v}^+(\alpha) - \chi(\varpi)^{-1} J_{\pi_v}^-(\alpha) ).
\end{align*}
Here each integral is given by
$$
J^+_{\pi_v}(\alpha) = \int_{S(F_v)}  \sH(p, s, w )^{-1} \psi_v(\alpha x) |t|^{-\frac{1}{2}} \chi(\alpha t) \ch_{\cO}(\alpha t)\, \mathrm dp,
$$
$$
J^-_{\pi_v}(\alpha) = \int_{S(F_v)}  \sH(p, s, w )^{-1} \psi_v(\alpha x) |t|^{-\frac{1}{2}} \chi(\alpha t)^{-1} \ch_{\cO}(\alpha t)\, \mathrm dp,
$$
where $S$ is the Zariski closure of $P$ in $X$.

This type of integral is studied by the second author and Tschinkel in \cite{TT12}. 
They studied height zeta functions of equivariant compactifications of $P$ under some geometric conditions.

The surface $S$ is isomorphic to $\bP^2 = \{c=0\} \subset \bP^3$. The boundary divisors are $E_1 = \{c=d=0\} = E \cap S$ and $D_1 = \{a = c = 0\} = \tilde{D}\cap S$. Let $\omega$ be a right invariant top degree form on $P$. Let $F = \{ b=c=0\} \subset \bP^3$. Then we have
$$
\mathrm{div}(\omega) = -D_1 - 2E_1, \quad \mathrm{div}(t) = D_1 - E_1, \quad \mathrm{div}(x) = F - E_1.
$$
We denote the Zariski closures of $S$, $F$, $D_1$, and $E_1$ in a smooth integral model $\mathcal X$ of $X$ over $\mathrm{Spec} \, \cO$ 
by $\mathcal S$, $\mathcal F$, $\mathcal D_1$, and $\mathcal E_1$ respectively.
They form integral models of $S$, $F$, $D_1$, and $E_1$.
Let $\rho : S(F_v) \rightarrow \mathcal S(\bF_v)$ be the reduction map mod $\varpi$. Then we have
$$
J^+_{\pi_v}(\alpha) = \sum_{r \in \mathcal S(\bF_v)} 
\int_{\rho^{-1}(r)}  \sH(p, s, w )^{-1} \psi_v(\alpha x) |t|^{-\frac{1}{2}} \chi(\alpha t) \ch_{\cO}(\alpha t)\, \mathrm dp,
:= \sum_{r \in \mathcal S(\bF_v)} J^+_{\pi_v}(\alpha, r).
$$
We analyze $J^+_{\pi_v}(\alpha, r)$ following \cite{TT12}. When $r \in G(\bF_v)$, we have
$$
\sum_{ r \in G(\bF_v)} J^+_{\pi_v}(\alpha, r) = \int_{G(\cO)} \chi(\alpha) \, \mathrm dp = \chi(\alpha).
$$
When $r \in \mathcal (\mathcal D_1 \setminus \mathcal E_1)(\bF_v)$, we have
\begin{align*}
J^+_{\pi_v}(\alpha, r) &= \chi(\alpha) \int_{\rho^{-1}(r)} \sH(p, s, w )^{-1} \psi_v(\alpha x) |t|^{-\frac{1}{2}} 
\chi(\alpha t) \ch_{\cO}(\alpha t)\, \mathrm dt^{\times} \mathrm dx,\\
&=\chi(\alpha) (1-q^{-1})^{-1}\int_{\rho^{-1}(r)} \sH(p, s, w )^{-1} \psi_v(\alpha x) |t|^{-\frac{1}{2}} \chi(t) \ch_{\cO}(\alpha t)\, \mathrm d|\omega|_v\\
&=\chi(\alpha) (1-q^{-1})^{-1}\int_{\rho^{-1}(r)} \sH(p, s-1, w-2 )^{-1} \psi_v(\alpha x) |t|^{-\frac{1}{2}} \chi( t) \ch_{\cO}(\alpha t)\, \mathrm d\tau_v
\end{align*}
where $\mathrm d\tau_v$ is the Tamagawa measure. Then there exist analytic local coordinates $y, z$ on $\rho^{-1}(r) \cong \frak m_v^2$ such that
\begin{align*}
J^+_{\pi_v}(\alpha, r) &= \chi(\alpha) (1-q^{-1})^{-1}\int_{\frak m_v^2} |y|_v^{s-1} |y|_v^{-\frac{1}{2}} \chi(y)\, \mathrm dy \mathrm dz,\\
&= \chi(\alpha) \frac{1}{q} \int_{\frak m_v} |y|_v^{s-\frac{1}{2}} \chi (y) \, \mathrm dy^{\times},\\
&= \chi(\alpha) \frac{1}{q}\sum_{m=1}^{+\infty} (q^{-(s-\frac{1}{2})} \chi(\varpi))^m,\\
&=\chi(\alpha)\frac{1}{q}\frac{q^{-(s-\frac{1}{2})} \chi(\varpi) }{1-q^{-(s-\frac{1}{2})}\chi(\varpi)}.
\end{align*}
If $r \in (\mathcal E_1 \setminus \mathcal D_1 \cup \mathcal F)(\bF_v)$, then there exist local analytic coordinates $y, z$ on $\rho^{-1}(r)$ such that
\begin{align*}
J^+_{\pi_v}(\alpha, r) &=\chi(\alpha)(1-q^{-1})^{-1} \int_{\frak m_v^2} |y|_v^{w-2} \psi_v(\alpha y^{-1}) 
|y|_v^{\frac{1}{2}}\chi(y)^{-1} \ch_{\cO}(\alpha y^{-1}) \, \mathrm dy \mathrm dz,\\
&= \chi(\alpha) \frac{1}{q} \int_{\frak m_v}
|y|_v^{w-\frac{1}{2}}\chi(y)^{-1} \ch_{\cO}(\alpha y^{-1}) \mathrm dy^{\times},\\
&= \chi(\alpha)\frac{1}{q}\sum_{m = 1}^k (q^{-(w-\frac{1}{2})} \chi(\varpi)^{-1} )^m.
\end{align*}
Suppose that $r \in \mathcal (D_1 \cap \mathcal E_1)(\bF_v)$. In this case there exist local analytic coordinates $y, z$ on $\rho^{-1}(r)$ such that
\begin{align*}
J^+_{\pi_v}(\alpha, r) &= \chi(\alpha) (1-q^{-1})^{-1} \int_{\frak m_v^2} |y|_v^{s-1} |z|_v^{w-2} \psi_v \left(\alpha \frac{1}{z}\right) 
|y/z|_v^{-\frac{1}{2}} \chi(y/z) \ch(\alpha y/z) \, \mathrm dy\mathrm dz\\
&= \chi(\alpha) (1-q^{-1}) \int_{\frak m_v^2} |y|_v^{s - \frac{1}{2}} |z|_v^{w-\frac{1}{2}} \psi_v \left(\alpha \frac{1}{z}\right) 
\chi(y/z) \ch(\alpha y/z) \, \mathrm dy^{\times}\mathrm dz^{\times}.
\end{align*}
Now we need the following lemma:
\begin{lemm}
$$
\int_{\cO^\times} \psi_v(\beta x) \, \mathrm dx^{\times} = \left\{
\begin{array}{ll}
1 & \quad \text{if $\beta \in \cO$}\\
-\frac{1}{q-1}  & \quad \text{if $\mathrm{ord}(\beta) = -1$}\\
0  & \quad \text{if $\mathrm{ord}(\beta) \leq -2$}
\end{array} \right.
$$
\end{lemm}
Using this lemma, we have
\begin{align*}
J^+_{\pi_v}(\alpha, r) &= \chi(\alpha) (1-q^{-1}) \sum_{m=1}^k (q^{-(w-\frac{1}{2})}\chi(\varpi)^{-1})^m \int_{\frak m_v} 
|y|_v^{s - \frac{1}{2}} \chi(y)  \, \mathrm dy^\times\\
& \,\,\,\,\,\,\,\,\,\,\,-\chi(\alpha)\frac{1}{q} (q^{-(w-\frac{1}{2})}\chi(\varpi)^{-1})^{k+1}\int_{\frak m_v} 
|y|_v^{s - \frac{1}{2}} \chi(y) \, \mathrm dy^\times \\
&=\chi(\alpha) (1-q^{-1}) \left(\sum_{m=1}^k (q^{-(w-\frac{1}{2})}\chi(\varpi)^{-1})^m \right)\frac{q^{-(s-\frac{1}{2})} \chi (\varpi) }{ 1-q^{-(s-\frac{1}{2})} \chi (\varpi) }\\
&\,\,\,\,\,\,\,\,\,\,\,-\chi(\alpha)\frac{1}{q} (q^{-(w-\frac{1}{2})}\chi(\varpi)^{-1})^{k+1} \frac{q^{-(s-\frac{1}{2})} \chi (\varpi) }{ 1-q^{-(s-\frac{1}{2})} \chi (\varpi) }.
\end{align*}
Now assume that $r \in (\mathcal E_1 \cap \mathcal F)(\bF_v)$. Then there exist local analytic coordinates such that
\begin{align*}
J^+_{\pi_v}(\alpha, r) &= \chi(\alpha) (1-q^{-1})^{-1} \int_{\frak m_v^2}|y|_v^{w-2} \psi_v(\alpha z/y) |y|_v^{\frac{1}{2}} 
\chi(y)^{-1} \ch_{\cO}(\alpha/y) \, \mathrm dy\mathrm dz\\
&= \chi(\alpha) \int_{\frak m_v^2}|y|_v^{w-\frac{1}{2}} \psi_v(\alpha z/y)  \chi(y)^{-1} \ch_{\cO}(\alpha/y) \, \mathrm dy^\times \mathrm dz\\
&= \chi(\alpha)\frac{1}{q}\sum_{m=1}^k (q^{-(w-\frac{1}{2})}\chi(\varpi)^{-1})^m 
\end{align*}
Putting everything together, we obtain the following
\begin{align*}
J^+_{\pi_v}(\alpha) &= \chi(\alpha) + \sum_{r \in (\mathcal D_1 \setminus \mathcal E_1)(\bF_v)} 
J^+_{\pi_v}(\alpha, r) +  \sum_{r \in (\mathcal E_1 \setminus (\mathcal D_1 \cup \mathcal F)(\bF_v)} J^+_{\pi_v}(\alpha, r) \\
& \,\,\,\,\,\,\,\,\,\,\,\,\,\,+  \sum_{r \in (\mathcal D_1 \cap \mathcal E_1)(\bF_v)} J^+_{\pi_v}(\alpha, r) 
+  \sum_{r \in (\mathcal F \cap \mathcal E_1)(\bF_v)} J^+_{\pi_v}(\alpha, r)\\
&= \chi(\alpha) ( 1 + \frac{q^{-(s-\frac{1}{2})} \chi(\varpi) }{1-q^{-(s-\frac{1}{2})}\chi(\varpi)} + \sum_{m = 1}^k (q^{-(w-\frac{1}{2})} \chi(\varpi)^{-1} )^m\\
 &\,\,\,\,\,\,\,\,\,\,\,+ (1-q^{-1}) \left(\sum_{m=1}^k (q^{-(w-\frac{1}{2})}\chi(\varpi)^{-1})^m \right)\frac{q^{-(s-\frac{1}{2})} \chi (\varpi) }{ 1-q^{-(s-\frac{1}{2})} \chi (\varpi) } \\
 &\,\,\,\,\,\,\,\,\,\,\,\,\,\,\,\,- \frac{1}{q} (q^{-(w-\frac{1}{2})}\chi(\varpi)^{-1})^{k+1} \frac{q^{-(s-\frac{1}{2})} \chi (\varpi) }{ 1-q^{-(s-\frac{1}{2})} \chi (\varpi) } )
\end{align*}
The same formula holds for $J^-_{\pi_v}(\alpha)$ by replacing $\chi$ with $\chi^{-1}$. 
From these expressions, we conclude the proof of the following lemma:
\begin{lemm}
Let $0 < \delta < \frac{1}{2}$ be a positive real number such that
$$
q^{-\delta} \leq |\chi(\varpi)| \leq q^\delta.
$$
Then the local integral $J_{\pi_v}(\alpha)$ is holomorphic in the domain $\Re (s) > \frac{1}{2} + \delta$.
\end{lemm}
We state an approximation of Ramanujan conjecture:
\begin{theo}\cite{LRS99}
 There exists a constant $0 < \delta < \frac{1}{2}$ such that
 for any non-archimedean place $v$ and any unramified principal series $\pi_v = \mathrm{Ind}_P^G(\chi \otimes \chi^{-1})$ arising
 as a local representation of an automorphic cuspidal representation $\pi$ in $L^2_{\mathrm{cusp}}(G(\bQ)\backslash G(\bA))$, we have
 \[
  q_v^{-\delta} <|\chi(\varpi)|< q_v^{\delta}. 
 \]
\end{theo}

We proceed to compute $J_{\pi_v}(\alpha)$.
\begin{align*}
&\chi(\varpi) J_{\pi_v}^+(\alpha) - \chi(\varpi)^{-1} J_{\pi_v}^-(\alpha) = \chi(\varpi\alpha) - \chi(\varpi\alpha)^{-1}\\
&\,\,\,\,\,\,\,\,\,+L(s-1/2, \pi)\left((\chi(\varpi^2\alpha) - \chi(\varpi^2\alpha)^{-1})q^{-(s-\frac{1}{2})} - (\chi(\varpi\alpha) - \chi(\varpi\alpha)^{-1}) q^{-2(s-\frac{1}{2})}\right)\\
&\,\,\,\,\,\,\,\,\,\,\,\,\,\,\,\,\,\,+ \sum_{m=1}^k q^{-m(w-\frac{1}{2})}\left(\chi(\varpi^{-m+1}\alpha) - \chi(\varpi^{-m+1}\alpha)^{-1}\right)\\
&\,\,\,\,\,\,\,\,\,\,\,\,\,\,\,\,\,\,\,\,\,+ (1-q^{-1})L(s-1/2, \pi) \sum_{m=1}^kq^{-(s-\frac{1}{2}) - m (w-\frac{1}{2})} \\
&\,\,\,\,\,\,\,\,\,\,\,\,\,\,\,\,\,\,\,\,\,\,\,\,\,\left( \chi(\varpi^{-m+2}\alpha) - \chi(\varpi^{-m+2}\alpha)^{-1} - q^{-(s-\frac{1}{2})} (\chi(\varpi^{-m+1}\alpha) - \chi(\varpi^{-m+1}\alpha)^{-1}\right)\\
&\,\,\,\,\,\,\,\,\,\,\,\,\,\,\,\,\,\,\,\,\,\,\,\,\,\,\,\,\,\,\,\,\, -\frac{1}{q}L(s-1/2, \pi)q^{-(s-\frac{1}{2}) - (k+1)(w-\frac{1}{2})}\left( \chi(\varpi^{-k+1}\alpha) - \chi(\varpi^{-k+1}\alpha)^{-1}\right)
\end{align*}
Now using Whittaker functions, we summarize these computations in the following way:
\begin{align*}
J_{\pi_v}(\alpha) &=  \frac{|\alpha|_v^{\frac{1}{2}}}{\chi(\varpi)-\chi(\varpi)^{-1}} (\chi(\varpi) J_{\pi_v}^+(\alpha) - \chi(\varpi)^{-1} J_{\pi_v}^-(\alpha) )\\
&= W_{\pi_v}(\alpha) +L(s-1/2, \pi)\left(q^{-(s-1)}W_{\pi_v}(\alpha\varpi) - q^{-2(s-\frac{1}{2})}W_{\pi_v}(\alpha) \right) 
+ \sum_{m=1}^k  q^{-mw}W_{\pi_v}(\alpha \varpi^{-m})\\
&+ (1-q^{-1}) L(s-1/2, \pi)\left(\sum_{m=1}^kq^{-(s-1)-mw} W_{\pi_v}(\alpha \varpi^{-m+1}) - q^{-2(s-\frac{1}{2}) -mw}W_{\pi_v}(\alpha \varpi^{-m})\right)\\
&\,\,\,\,\,\,\,\,\,\,\,\,\,\,\,\,\,\,-L(s-1/2, \pi) q^{-s-(k+1)w}
\end{align*}
To obtain an estimate of this integral, we use the following lemma:
\begin{lemm}
\label{lemm: whittaker estimates}
Let $0 < \delta < \frac{1}{2}$ be a positive real number such that
$$
q^{-\delta} \leq |\chi(\varpi)| \leq q^\delta.
$$
Then we have
$$
|W_{\pi_v}(\varpi^m)| \leq 2mq^{-m(\frac{1}{2}-\delta)}
$$
\end{lemm}

We come to the conclusion of this section. We let
\[
 \Lambda = \{(x,y) \in \bR^2 \mid x > 1, \quad x+y >0\},
\]
and 
\[
 \mathsf T_\Lambda = \{(s, w) \in \bC^2 \mid (\Re(s), \Re(w)) \in \Lambda \}.
\]

\begin{lemm}
\label{lemm:estimate}
Let $K \subset \Lambda$ be a compact subset and $\alpha \in \cO\setminus \cO^\times$. 
Then there exists a constant $C_K >0$ which does not depend on $v$, $\alpha$ and $\pi_v$ such that
$$
|J_{\pi_v}(\alpha)| \leq C_K v(\alpha) |\alpha|^\rho\frac{|L(s-1/2, \pi_v)|}{|\zeta(s+w)|},
$$
on $\mathsf T_K$ where $\rho = \max \{-\Re(w) \mid (s, w)\in \mathsf T_K \}$.
\end{lemm}

\begin{proof}
Suppose that $(s, w) \in \mathsf T_K$. Let $k = v(\alpha) >0$. We apply Lemma~\ref{lemm: whittaker estimates} to bound $J_{\pi_v}(\alpha)$:
\begin{align*}
|J_{\pi_v}(\alpha)| &\leq |L(s-1/2, \pi_v)||\alpha|^\rho(4|W_{\pi_v}(\alpha)| +|W_{\pi_v}(\alpha \varpi)| + |W_{\pi_v}(\alpha)| 
+ 4\sum_{m=1}^k |W_{\pi_v}(\alpha \varpi^{-m})|\\
&\,\,\,\,\,\,\,\,\,\,\,\,\,\,\,\,\,+\sum_{m=1}^k\left( |W_{\pi_v}(\alpha \varpi^{-m+1})|+ |W_{\pi_v}(\alpha \varpi^{-m})|\right) +1)\\
&\leq |L(s-1/2, \pi_v)| |\alpha|^\rho(C_1 + 5\sum_{m=1}^k 2(k-m)q^{-(k-m)(\frac{1}{2} - \delta)} + \sum_{m=1}^k 2(k+1-m)q^{-(k+1-m)(\frac{1}{2} - \delta)}),
\end{align*}
where $C_1$ is a constant which does not depend on $v$ or $\alpha$. Let $0 < \epsilon < \frac{1}{2} - \delta$. Then there exists another constant $C_2$ such that
\begin{align*}
&5\sum_{m=1}^k 2(k-m)q^{-(k-m)(\frac{1}{2} - \delta)} + \sum_{m=1}^k 2(k+1-m)q^{-(k+1-m)(\frac{1}{2} - \delta)}\\
&\leq C_2 \sum_{m=1}^k q^{-(k-m)\epsilon}   \leq C_2k
\end{align*}
Hence our assertion follows.
\end{proof}

\subsection{$v$ archimedean}
\label{subsec: cusp_archimedean}

Let $\pi_\infty$ be an unramified principal series $\mathrm{Ind}_P^G(|.|^\mu \otimes |.|^{-\mu})$ with a $K_\infty$-fixed vector for some complex number $\mu$. 
Its Whittaker function satisfies
$$
W_{\pi_\infty}\left(\begin{pmatrix}1&x \\ & 1 \end{pmatrix}\begin{pmatrix}t& \\ & 1 \end{pmatrix}
\begin{pmatrix}\cos\theta&-\sin\theta \\ \sin \theta& \cos \theta \end{pmatrix}\right) = \psi_\infty(x) 
W_{\pi_\infty}\begin{pmatrix}t& \\ & 1 \end{pmatrix}
$$
where $\psi_\infty : \bR \rightarrow \mathbb S^1$ is the standart additive character, and
the function $W_{\pi_\infty}\begin{pmatrix}t& \\ & 1 \end{pmatrix}$ is given by
$$
W_{\pi_\infty}\begin{pmatrix}t& \\ & 1 \end{pmatrix} = 2\frac{\pi^{\mu +1/2} |t|^{1/2}}{\Gamma (\mu + 1/2)} K_{\mu}(2\pi|t|),
$$
where $K_\mu(t)$ is the modified Bessel function of the second kind.
See \cite[Proposition 7.3.3]{GH11}.
However, later we normalize the Whittaker function at the archimedean place by $W_{\pi_\infty}(e) = 1$, so we may assume that
\[
 W_{\pi_\infty}\begin{pmatrix}t& \\ & 1 \end{pmatrix} = \frac{|t|^{1/2}}{K_\mu(2\pi)} K_{\mu}(2\pi|t|)
\]

The fact that $K_\mu(t)$ decays exponentially as $t \rightarrow \infty$ gives us the following lemma:
\begin{lemm}
The local height integral $J_{\pi_\infty}(\alpha)$ is holomorphic in the domain defined by $\Re(s) >1$ and $\Re(s+w) >0$.
\end{lemm}
\begin{proof}
Proposition 4.2 of \cite{volume} shows that $J_{\pi_\infty}(\alpha)$ is holomorphic in $\Re(s) > 1$ and $\Re(w) > 0$. 
We extend this domain by using the rapidly decaying function $K_\mu$.
Since our height is invariant under the action of $K_\infty = \mathrm{SO}_2(\bR)$, using the integration formula we have
\begin{align*}
J_{\pi_\infty}(\alpha) &= \int_{G(\bR)} W_{\pi_\infty}\left(\begin{pmatrix}\alpha & \\ & 1 \end{pmatrix} g\right)\sH(g, s, w)^{-1} \, \mathrm dg \\
&= \int_{P(\bR)} W_{\pi_\infty}\left(\begin{pmatrix}\alpha & \\ & 1 \end{pmatrix} p\right)\sH(p, s, w)^{-1} \, \mathrm d_lp\\
&= \int_{S(\bR)} W_{\pi_\infty}\left(\begin{pmatrix}\alpha t  & \\ & 1 \end{pmatrix}\right) \sH(p. s, w)^{-1} \psi_\infty(\alpha x) |t|^{-1} \, 
\mathrm dx \mathrm dt^\times,
\end{align*}
where $S$ is the Zariski closure of $P$ in $X$.
Let $\{U_\eta\}$ be a sufficiently fine finite open covering of $S(\bR)$ and consider a partition of unity $\theta_\eta$ subordinate to this covering.
Using this partition of unity, we obtain
\begin{equation}
\label{eqn:partition}
J_{\pi_\infty}(\alpha) = \sum_\theta \int_{S(\bR)} W_{\pi_\infty}\left(\begin{pmatrix}\alpha t  & \\ & 1 \end{pmatrix}\right) \sH(p. s, w)^{-1} 
\psi_\infty(\alpha x) |t|^{-1} \theta_\eta \, \mathrm dx \mathrm dt^\times.
\end{equation}
Suppose that $U_\eta$ meets with $E_1$, but not $D_1$. Then the term corresponding to $\eta$ in \eqref{eqn:partition} looks like
$$
\int_{\bR^2} \Phi(y, z, s, w) |y|^{w-1} W_{\pi_\infty}(\alpha y^{-1})\, \mathrm dy \mathrm dz,
$$
where $\Phi(y, z, s, w)$ is a bounded function with a compact support. 
Since the Whittaker function decays rapidly, this function is holomorphic everywhere.
Note that $\alpha$ is non-zero.

Next assume that $U_\eta$ contains the intersection of $E_1$ and $D_1$. In this case, the term in \eqref{eqn:partition} looks like
$$
\int_{\bR^2} \Phi(y, z, s, w) |z|^{s-1}|y|^{w-1} W_{\pi_\infty}(\alpha z/y)\, \mathrm dy \mathrm dz,
$$
Applying a change of variables by $z = z'$ and $y = y'z'$, it becomes
$$
\int_{\bR^2} \Phi(y'z', z', s, w) |z'|^{s+w-1}|y'|^{w-1} W_{\pi_\infty}(\alpha/y')\, \mathrm dy \mathrm dz.
$$
This integral is absolutely convergent if $\Re(s+w) >0$.
\end{proof}

\begin{lemm}
\label{lemm:vectorfields}
Let $\partial^X$ be a left invariant differential operator on $X$. Then the function
$$
\frac{\partial^X (\sH_\infty (g, s, w)^{-1})}{\sH_\infty (g, s, w)^{-1}}
$$
is a smooth function on $X(\bR)$.
\end{lemm}
\begin{proof}
See the proof of Proposition 2.2 in \cite{VecIII}.
\end{proof}
Here we use an iterated integration idea. Let 
$$
h= \begin{pmatrix} 1\\ & -1 \end{pmatrix}, \,\,\, v_+ = \begin{pmatrix} 0 & 1 \\ 0 & 0 \end{pmatrix}, \,\,\, v_- = \begin{pmatrix} 0 & 0 \\ 1 & 0 \end{pmatrix}, 
$$
and think of them as elements of the universal enveloping algebra of the complexified Lie algebra of $\PGL_2(\bR)$. The Casimir operator is given by 
$$
\Omega = \frac{h^2}{4} - \frac{h}{2} + v_+ v_-. 
$$
Then we have 
$$
\Omega.W_{\pi_\infty} = \lambda_\pi W_{\pi_\infty}. 
$$
As a result for any integrable smooth function $f$ such that $f$ is right $K_\infty$-invariant and
and its iterated derivatives are also integrable, we have 
\begin{align*}
\int_{G(\bR)} f(g) W_{\pi_\infty}\left(\begin{pmatrix}\alpha & \\ & 1\end{pmatrix}g\right) \, \mathrm dg 
&= \lambda_\pi^{-N} \int_{G(\bR)} \Omega^N f(g) W_{\pi_\infty}\left(\begin{pmatrix}\alpha & \\ & 1\end{pmatrix}g\right) \,\mathrm dg
\\&= \lambda_\pi^{-N} \int_{P(\bR)} \Omega^N f(p) W_{\pi_\infty}\begin{pmatrix} \alpha t & \\ & 1 \end{pmatrix}\psi(\alpha x) \, \mathrm d_lp\\
&= \lambda_\pi^{-N} \alpha^{-M} \int_{P(\bR)} \left(\frac{\partial }{\partial x}\right)^M(\Omega^N f(p)) W_{\pi_\infty}\begin{pmatrix} \alpha t & \\ & 1 \end{pmatrix}\psi_\infty(\alpha x) \,\mathrm d_lp
\end{align*}
Note that $\Omega^N f(g)$ is right $K_\infty$-invariant because $\Omega$ is an element of the center of the universal enveloping algebra. Hence we have
$$
\left|\int_{G(\bR)} f(g) W_{\pi_\infty}(g) \, dg \right| \leq \lambda_\pi^{-N} \alpha^{-M}\left\| \left(\frac{\partial }{\partial x}\right)^M\Omega^Nf \right\|_1 \cdot \| W_{\pi_\infty} \|_\infty. 
$$
We will apply this simple idea to our height function
$$
\sH(g, s, w)^{-1}. 
$$
We define the domain $\Lambda$ by
\[
 \Lambda = \{(x,y) \in \bR^2 \mid x >1, \quad x+y > 2\}.
\]

\begin{lemm}
Fix positive integers $N$ and $M$. Let $K$ be a compact set in $\Lambda$.
Then there exists a constant $C$ depending on $N, M, K$ such that we have
\[
\left| \left(\frac{\partial }{\partial x}\right)^M \Omega^N \sH(g, s, w)^{-1} \right| < C\sH(g, \Re(s), \Re(w))^{-1},
\]
for all $g\in G(\bR)$ and $s, w$ such that $(s, w) \in \mathsf T_K$.
\end{lemm}
\begin{proof}
First note that $\partial / \partial x$ is a RIGHT invariant differential operator on $P = \left\{ \begin{pmatrix}a&x \\ &1 \end{pmatrix}\right\}$. Moreover the surface $S$ is a biequivariant compactfication of $P$ so that the differential operator $\partial / \partial x$ extends to $S$. Now our assertion follows from Lemma~\ref{lemm:vectorfields}.
\end{proof}
%In order to see this, we proceed as follows. A computation shows that if $g = \begin{pmatrix} a & b \\ c & d \end{pmatrix}$, $u = w-s$, $v = -s-w$, then 
%$$
%h.f(g) =  \left( 2u \frac{c^2 - d^2}{c^2 + d^2} + 2v \frac{a^2 - b^2 + c^2 - d^2}{a^2 + b^2 + c^2 + d^2 }\right)f(g), 
%$$
%$$
%v_+ f (g) = v_-f(g) = \left(2u \frac{cd}{c^2 + d^2} + 2v \frac{ab + cd}{a^2 + b^2 + c^2 + d^2 } \right)f(g). 
%$$
%Next let 
%$$
%R(g) = \frac{A(c, d)}{B(c, d)} \,\,\,\,\text{ or } R(g) = \frac{A(a, b, c, d)}{B(a, b, c, d)}
%$$
%with $A$ a polynomial, and $B(c, d) = (c^2 + d^2)^m$ and $B(a, b, c, d) = (a^2 + b^2 + c^2 + d^2)^m$ with $\deg A \leq 2m$ in both cases.  First we observe that any function of this form is bounded. Furthermore, each of the functions $h.R$, $v_+.R$, and $v_-.R$ will be a function of the same shape. This is sufficient to give the lemma. 

\

Combining these statements with the computation of $\| \sH(g, s, w)^{-1} \|_1 = \langle \sZ, 1 \rangle_\infty$ gives the following lemma:
\begin{lemm}
\label{lemm:archimedean}
Fix positive integers $N$ and $M$. Let $K$ be a compact set in $\Lambda$.
Then there exists a constant $C$ only depending on $N, M, K$, but not $\pi_\infty$ and $\alpha$ such that
$$
|J_{\pi_\infty}(\alpha)| <  C\lambda_\pi^{-N} \alpha^{-M}  
$$
whenever $(s,w) \in \mathsf T_K$. 
\end{lemm}
\begin{proof}
 We need to obtain an upper bound for the integral
 \[
  \int_{P(\bR)} \sH(g, \Re(s), \Re(w))^{-1} W_{\pi_\infty}\begin{pmatrix}\alpha t & \\ & 1\end{pmatrix} \, \mathrm dp_l.
 \]
 First note that by an approximation of Ramanujan conjecture, there exists $0 < \delta < 1/2$ such that $|\Re(\mu)| \leq \delta$.
 Let $\epsilon = \min \{\Re(s+w) - 2 \mid (s, w) \in \mathsf T_K\} > 0$.
 
 Suppose that $\Re(w) \geq \epsilon/2$. Then the above integral is bounded by
 \[
  \sqrt{\pi} \frac{\Gamma(\frac{\Re(s)-1}{2})\Gamma(\frac{\Re(w)}{2})}{\Gamma(\frac{\Re(s+w)}{2})} \left\|W_{\pi_\infty}\right\|_\infty.
 \]
 It follows from results in Section~\ref{sec: appendix} that we have
 \[
 \left\|W_{\pi_\infty}\right\|_\infty \ll |\Im(\mu)|^2.
 \]
 Finally note that $\lambda_\pi = (1/4 - \mu^2)$, so our assertion follows.
 
 Next suppose that there exists a positive integer $m$ such that $\epsilon/2 \leq \Re(w) + m \leq 1+ \epsilon/2$.
 In this situation, we have $\Re(s) -m \geq 1+\epsilon/2$.
 It follows from Lemma~\ref{lemm: integral_computation} that 
 \[
  \left|\int_{P(\bR)} \sH(g, \Re(s), \Re(w))^{-1} W_{\pi_\infty}\begin{pmatrix}\alpha t & \\ & 1\end{pmatrix} \, \mathrm dp_l\right|
  \leq |\alpha|^{-m}\sqrt{\pi} \frac{\Gamma(\frac{\Re(s)-m-1}{2})\Gamma(\frac{\Re(w)+m}{2})}{\Gamma(\frac{\Re(s+w)}{2})}
  \left\|t^mW_{\pi_\infty}(t)\right\|_\infty.
 \]
 Again it follows from Section~\ref{sec: appendix} that 
 \[
  \left\|t^mW_{\pi_\infty}(t)\right\|_\infty \ll |\Im(\mu)|^{m+2}.
 \]
 Thus our assertion follows.
\end{proof}

\subsection{The adelic analysis}
\label{subsec: cusp_adelic}
As $\sZ$ is right $K$-invariant, 
the only  automorphic cuspidal representations that contribute to the automorphic Fourier expansion of $\sZ$ are those 
that have a $K$-fixed vector. 
Let $\pi = \otimes_v' \pi_v$ be an automorphic cuspidal representation of $\PGL(2)$.  
By a theorem of Jacquet-Langlands every component $\pi_v$ is generic.
See \cite[Chapter 2, Proposition 9.2]{JL70}.
Let $W_{\pi_v}$ be the $K_v$-fixed vector in the space of $\pi_v$ normalized so that $W_{\pi_v}(e) = 1$. 
Let $\phi_\pi$ be the $K$-fixed vector in the space of $\pi$ normalized so that $\langle \phi_\pi, \phi_\pi \rangle =1$.
Then 
$$
W_{\phi_\pi} = W_{\phi_\pi} (e)\cdot \prod_v W_{\pi_v}. 
$$
 Note that 
 $$
 W_{\phi_\pi} (e) \ll \| \phi_\pi \|_\infty. 
 $$
 Indeed, we have
$$
W_{\phi_\pi}(g)  = \int_{ \bQ \backslash \bA} \phi \left(\begin{pmatrix} 1 & x \\ & 1 \end{pmatrix} g\right) \, \psi (-x) \, \mathrm dx. 
$$
Hence we conclude
$$
|W_{\phi_\pi}(g)| \leq \int_{ \bQ \backslash \bA} \left| \phi \left(\begin{pmatrix} 1 & x \\ & 1 \end{pmatrix} g\right) \right|  \, \mathrm dx 
\leq \| \phi \|_{\infty}  \, \mathrm{vol} (\bQ \backslash \bA).
$$
We have 
$$
\langle \sZ(s, .), \phi_\pi \rangle 
= \sum_{\alpha \in F^\times} \int_{G(\bA)}W_{\phi_\pi}\left(\begin{pmatrix} \alpha \\ & 1 \end{pmatrix} g\right) \sH(\mathbs, g)^{-1}\, \mathrm dg
$$ 
$$
= 2W_{\phi_\pi}(e) \sum_{\alpha=1}^\infty  \prod_v J_{\pi_v}(\alpha). 
$$
Note that for any non-archimedean place $v$, we have $J_{\pi_v}(\alpha) = 0$ if $\alpha \not\in \cO_v$.
Also note that $J_{\pi_v}(\alpha) = J_{\pi_v}(-\alpha)$.
\begin{lemm}
The series 
$$
J_\pi=\sum_{\alpha=1}^\infty  \prod_{v \leq \infty} J_{\pi_v}(\alpha)
$$
is absolutely convergent for $\Re(s)  >3/2 +\epsilon+ \delta$ and $\Re(s+w) > 2+\epsilon$ for any real number $\epsilon >0$.  
Furthermore, let $\Lambda= \{(x,y)\in \bR^2 \mid x >3/2 + \delta, \quad x+y > 2\}$ and $K$ be a compact set in $\Lambda$.
Then there exists a constant $C_{K, N}>0$ independent of $\pi$ such that 
\[
|J_\pi| \leq \lambda_\pi^{-N}C_{K, N}\left|\frac{L(s-1/2, \pi)}{\zeta(s+w)}\right|,
\]
for any $(s,w) \in \mathsf T_K$.
\end{lemm}
\begin{proof}
To see this, define a multiplicative function $F(\alpha)$ by 
$$
F(p^k) = C_1k
$$
where $C_1$ is a constant in Lemma~\ref{lemm:estimate}.
By Lemma~\ref{lemm:estimate} and Lemma ~\ref{lemm:archimedean}, we know that 
$$
|\prod_{v \leq \infty} J_{\pi_v}(\alpha)| \leq \left| C_2 \lambda_\pi^{-N}\alpha^{-M}\frac{L(s-1/2, \pi)}{\zeta(s+w)}\right| F(\alpha).
$$
Thus we need to discuss the convergence of $\sum_\alpha \frac{F(\alpha)}{\alpha^M}$. Formally this infinite sum is given by
$$
\prod_p\left( 1+ \sum_{k=1}^\infty \frac{C_1k}{p^{kM}}\right).
$$
Then the product
$$
\prod_p\left( 1+ \frac{C_1}{p^{M}}\right)
$$
absolutely converges as soon as $M > 1$. Thus we need to show the convergence of
\begin{equation}
\label{eqn:infinite sum}
\sum_p \sum_{k=2}^\infty  \frac{k}{p^{kM}}.
\end{equation}
Let $\epsilon >0$ be a sufficiently small positive real number. Then there exists a constant $C_3$ such that
$$
\frac{k}{p^{k \epsilon}} \leq C_3 
$$
for any $k$ and $p$. Then the infinite sum~\eqref{eqn:infinite sum} is bounded by
$$
C_3\sum_p \sum_{k=2}^\infty  \frac{1}{p^{k(M-\epsilon)}} = C_3\sum_p \frac{p^{-2(M-\epsilon)}}{1 - p^{-(M-\epsilon)}} < C(1-2^{-(M-\epsilon)})^{-1} \sum_p p^{-2(M-\epsilon)}.
$$
The last infinite sum converges if $M >1$ and $\epsilon$ is sufficiently small.
\end{proof}

Now look at the cuspidal contribution 

$$
Z(s)_{\mathrm{cusp}} = Z(s, e)_{\mathrm{cusp}} = \sum_\pi \langle \sZ(s, .), \phi_\pi\rangle \phi_\pi(e). 
$$

\begin{lemm}
\label{lemm: cuspidalpart}
The series $Z(s)_{\mathrm{cusp}}$ is absolutely uniformaly convergent for $\Re(s) > 3/2+\delta$ and $\Re (s+w) > 2$, 
and defines a holomorphic function in that region. 
\end{lemm}
\begin{proof}
We have 
\begin{align*}
\sum_\pi & |\langle \sZ(s, .), \phi_\pi\rangle|\cdot |\phi_\pi(e)| \\
& \leq \sum_\pi  |2W_{\phi_\pi}(e) \phi_\pi(e)| \sum_{\alpha = 1}^\infty |J_{\pi_\infty}(\alpha)|. \left| \prod_{v < \infty} J_{\pi_v}(\alpha) \right| \\
& \ll  \sum_\pi  \| \phi_\pi \|_\infty^2\cdot \lambda_\pi^{-N}  \left|\frac{L(s-1/2, \pi)}{\zeta(s+w)}\right| \\
& \ll C_\delta  \sum_\pi  
\| \phi_\pi \|_\infty^2 \cdot \lambda_\pi^{-N} |L(s-1/2, \pi)| \\
& \ll C_\delta 
\sum_\pi  \| \phi_\pi \|_\infty^2 \cdot \lambda_\pi^{-N} 
\end{align*}
by the rapid decay of the $K$-Bessel function. 
This last series converges for $N$ large. See Theorem 7.4 of \cite{PGL2}. 
\end{proof}

\section{Step three: Eisenstein contribution}\label{sect:eisenstein}

For ease of reference we set 
$$
E(s, g)_{\mathrm{c}} = \chi_P(g)^{s+ 1/2} + \frac{\Lambda(2s)}{\Lambda(2s+1)} \chi_P(g)^{-s + 1/2}
$$
and 
$$
E(s, g)_{\mathrm{nc}}= \frac{1}{\zeta(2s+1)}\sum_{\alpha \in \bQ^\times} W_s\left(\begin{pmatrix} \alpha \\ & 1 \end{pmatrix} g \right),
$$
where $s\in \mathbb C$, $g \in G(\bA)$, and $\chi_P, W_s$ are introduced in Section~\ref{subsect:eisenstein}.
Note that we have the Fourier expansion of the Eisenstein series:
\[
E(s, g) =  E(s, g)_{\mathrm{c}} + E(s, g)_{\mathrm{nc}}.
\]

We will be considering integrals of the form 
$$
\sZ(s,g)_{\mathrm{eis}} : = \frac{1}{4\pi } \int_\bR \left( \int_{G(\bQ) \backslash G(\bA)}\sZ(s, h) \overline{E(it, h)} \, \mathrm dh \right)E(it, g) \, \mathrm dt. 
$$
First we examine the inner integral. As usual we have 
$$
\int_{G(\bQ) \backslash G(\bA)}\sZ(s, h) \overline{E(it, h)} \mathrm dh = 
 \int_{G(\bA)}E(-it, g)_{\mathrm{c}}  \sH(\mathbs, g)^{-1} \, \mathrm dg
 $$
 $$
 + \int_{G(\bA)}E(-it, g)_{\mathrm{nc}}  \sH(\mathbs, g)^{-1} \, \mathrm dg.
$$

\subsection{The non-constant term}
We let 
\begin{align*}
I_{\mathrm{nc}}(\mathbs, t) &= \int_{G(\bA)}E(-it, g)_{\mathrm{nc}}  \sH(\mathbs, g)^{-1} \, \mathrm dg \\
&= \frac{1}{\zeta(-2it+1)}\sum_{\alpha \in \bQ^\times} \prod_v \int_{G(\bQ_v)} W_{-it, v}\left(\begin{pmatrix} \alpha \\ & 1 \end{pmatrix} g_v \right)
\sH_v(\mathbs, g_v)^{-1} \, \mathrm dg_v.
\end{align*}
We define
\[
 J_{t, v}(\alpha) = \int_{G(\bQ_v)} W_{-it, v}\left(\begin{pmatrix} \alpha \\ & 1 \end{pmatrix} g_v \right)
\sH_v(\mathbs, g_v)^{-1} \, \mathrm dg_v.
\]

\

For any non-archimedean place $v$, $W_{s, v}$ satisfies the exactly same formula
for the Whittaker function $W_{\pi_v}$ of a local unramified principal seires $\pi_v = \mathrm{Ind}_P^G(\chi \otimes \chi^{-1})$
by replacing $\chi$ by $|\cdot|_v^s$, i.e., we have
\[
 W_{s, v}\left(\begin{pmatrix}1 & x \\ & 1\end{pmatrix}\begin{pmatrix}a &  \\ & 1\end{pmatrix} k\right)
 = \psi_v (x)  W_{s, v} \begin{pmatrix} a &  \\ & 1\end{pmatrix}
\]
for any $x \in \bQ_v$, $a \in \bQ^\times$, and $k \in K_v$. Moreover we have
\[
 W_{s, v} \begin{pmatrix} \varpi^m &  \\ & 1\end{pmatrix} = \begin{cases} q^{-m/2} \sum_{k=0}^m |\varpi^k|_v^s |\varpi^{m-k}|_v^{-s} & m \geq 0 ; \\ 
0 & m < 0.
\end{cases}
\]
Hence the computation in Section~\ref{subsec: cusp_non-archimedean} can be applied to $J_{t, v}(\alpha)$ without any modification.
We summarize this fact as the following lemma:
\begin{lemm}
 Let $v$ be a non-archimedean place.
 The function $J_{t, v}(\alpha)$ is holomorphic when $\Re(s) > \frac{1}{2}$.
 If $\alpha \not\in \cO_v$, then $J_{t, v}(\alpha) = 0$.
 If $\alpha \in \cO_v^\times$, then we have
 \[
  J_{t, v}(\alpha) = \frac{\zeta_v(s+it -1/2)\zeta_v(s-it-1/2)}{\zeta_v(s+w)}.
 \]
 In general, let $\Lambda = \{(x,y)\in \bR^2 \mid x > 3/2, \quad x+y >2\}$ and $K$ be a compact subset in $\Lambda$. 
 We define $\rho = \max \{-\Re(w) \mid (s,w) \in \mathsf T_K\}$.
 Tnen there exists a constant $C_K >0$ which does not depend on $t, v, \alpha$ such that
 \[
  |J_{t, v}(\alpha)| < C v(\alpha)|\alpha|^\rho \frac{|\zeta_v(s+it -1/2)\zeta_v(s-it-1/2)|}{|\zeta_v(s+w)|}.
 \]

\end{lemm}

For the archimedean place $v = \infty$, the Whittaker function is given by
\[
 W_{s, \infty}\begin{pmatrix}a& \\ & 1 \end{pmatrix} = 2\frac{\pi^{s+1/2 } |a|^{1/2}}{\Gamma (s+1/2)} K_{s}(2\pi|a|).
 \]
Moreover for the Casimir operator $\Omega$, we have
\[
 \Omega W_{s, \infty} = \left(\frac{1}{2}+ s\right)\left(\frac{1}{2}- s\right)W_{s, \infty}.
\]

As the discussion of Section~\ref{subsec: cusp_archimedean}, we conclude
\begin{lemm}
The function $J_{t, \infty}(\alpha)$ is holomorphic in the domain $\Re(s) >1$ and $\Re(s+w)>0$.
Fix positive integers $N$ and $M$. Let $K$ be a compact set in $\Lambda$.
Then there exists a constant $C_{N, M, K}$ only depending on $N, M, K$, but not $t$ and $\alpha$ such that
$$
|J_{t, \infty}(\alpha)| <  C_{N, M, K}(1+t^2)^{-N} \alpha^{-M} 
$$
whenever $(s,w) \in \mathsf T_K$.
\end{lemm}

Finally the discussion of Section~\ref{subsec: cusp_adelic} leads to the following lemma:
\begin{lemm}
 The series $I_{\mathrm{nc}}(\mathbs, t)$
is absolutely convergent for $\Re(s)  >3/2 +\epsilon$ and $\Re(s+w) > 2+\epsilon$ for any real number $\epsilon >0$.  
Furthermore, let $K$ be a compact subset in $\Lambda$.
Then there is a $C_{K, N}>0$ independent of $t$ such that 
\[
|I_{\mathrm{nc}}(\mathbs, t)| \leq (1+t^2)^{-N}C_{K, N}\frac{\zeta(\Re(s)-1/2)^2}{|\zeta(-2it + 1)\zeta(s+w)|}.
\]
\end{lemm}
The results of \S 4 and \S 5 of \cite{PGL2} show that 
$$
\int_\bR I_{\mathrm{nc}}(\mathbs, t)\cdot E(it, g) \, \mathrm dt 
$$
is holomorphic in the domain $\Re(s) >3/2$ and $\Re(s+w)>2$.
Note that it follows from Theorem~\ref{theo: zeta} that $|\zeta(1+2it)|^{-1} \ll \log |t|$,
so this does not affect convergence.

\subsection{The constant term}
\label{subsec: constant}
The issue is now understanding 
$$
I_{\mathrm{c}}(\mathbs, t) = \int_{G(\bA)}E(-it, g)_{\mathrm{c}}  \sH(\mathbs, g)^{-1} \, \mathrm dg
$$
$$
= \int_{G(\bA)} \chi_P(g)^{-it+ 1/2} \sH(\mathbs, g)^{-1} \, \mathrm dg + 
\frac{\Lambda(-2it)}{\Lambda(-2it+1)} \int_{G(\bA)} \chi_P(g)^{it + 1/2}\sH(\mathbs, g)^{-1} \, \mathrm dg. 
$$
$$
= \frac{\Lambda(s-it -1/2)\Lambda(w+it -1/2)}{\Lambda(s+ w)} + \frac{\Lambda(-2it)}{\Lambda(-2it+1)}  \frac{\Lambda(s+it -1/2)\Lambda(w-it -1/2)}{\Lambda(s+ w)}. 
$$
The last equality follows from Lemma~\ref{lemm: integral_computation}.
 We then have 
$$
\int_{\bR} I_{\mathrm c}(\mathbs, t)\cdot E(it, g) \, \mathrm dt = 2 \int_\bR \frac{\Lambda(s-it -1/2)\Lambda(w+it -1/2)}{\Lambda(s+ w)} E(it, g) \, \mathrm dt 
$$
after using the functional equation for the Eisenstein series. 
We denote the vertical line defined by $\Re(z) = a$ by $(a)$ for any real number $a \in \bR$.
Assume $\Re s, \Re w \gg 0$.  We then have 
$$
\frac{1}{4\pi} \int_{\bR} I_{\mathrm{c}}(\mathbs, t)\cdot E(it, g) \, \mathrm dt = \frac{2}{4 \pi i} 
\int_{(0)} \frac{\Lambda(s-y -1/2)\Lambda(w+y -1/2)}{\Lambda(s+ w)} E(y, g) \, \mathrm dy 
$$
$$
=  - \frac{\Lambda(s-1)\Lambda(w)}{\Lambda(s+ w)} \, {\rm Res}_{y = 1/2} E(y, g) 
+ \frac{\Lambda(s+w-2)}{\Lambda(s+ w)} E(s-3/2, g) -   \frac{\Lambda(s+w-1)}{\Lambda(s+ w)} E(s-1/2, g)
$$
$$
+ \frac{2}{4 \pi i} \int_{(L)} \frac{\Lambda(s-y -1/2)\Lambda(w+y -1/2)}{\Lambda(s+ w)} E(y, g) \, \mathrm dy
$$
by shifting the contour to $L + i \R$, for an $L > \Re s$, and picking up the residue at $y=1/2, s-1/2$, and  $s-3/2$.  The latter integral converges absolutely for $L \gg 0$. In fact, 
$$
\int_{(L)} \left| \frac{\Lambda(s-y -1/2)\Lambda(w+y -1/2)}{\Lambda(s+ w)} E(y, g) \right| \, \mathrm dy
$$
$$
\ll \frac{|E(L, g)|}{|\Lambda(s+ w)|}. \int_{(L)} \left| \Lambda(s-y -1/2)\Lambda(w+y -1/2) \right| \, \mathrm dy
$$
$$
=\frac{|E(L, g)|}{|\Lambda(s+ w)|}. \int_{(L)} \left| \Lambda(y +1/2-s)\Lambda(w+y -1/2)  \right| \, \mathrm dy
$$
$$
\ll \frac{|E(L, g)|\zeta(L+1/2 - \Re s) \zeta(L - 1/2 + \Re w)}{|\Lambda(s+ w)|}  
\int_{(L)} \left| \Gamma(\frac{y +1/2-s}{2})\Gamma(\frac{w+y -1/2}{2})  \right| \, \mathrm dy. 
$$
Note that $|E(y,g)| \leq E(\Re y, g).$
In order to see the convergence of the last integral we rewrite it as 
$$
 \int_{\R} \left| \Gamma(\frac{L +1/2-\Re s}{2} + i \frac{t -\Im s}{2} )\Gamma(\frac{L + \Re w -1/2}{2} + i \frac{t + \Im w }{2}  )  \right| \, \mathrm d t. 
$$
This last integral is easily seen to be convergent by Stirling's formula. 

\section{Step four: spectral expansion} \label{sect:spectral}
We start by fixing a basis of right $K$-fixed functions for $L^2(G(\bQ) \backslash G(\bA))$. We write 
$$
L^2(G(\bQ) \backslash G(\bA))^K = L_{\mathrm{res}}^K \oplus L_{\mathrm{cusp}}^K \oplus L_{\mathrm{eis}}^K. 
$$
Since $\bQ$ has class number one, $L_{\mathrm{res}}^K$ is the trivial representation.
An orthonormal basis of this space is the constant function 
$$
\phi_{\mathrm{res}}(g) = \frac{1}{\sqrt{\mathrm{vol} \, (G(\bQ)\backslash G(\bA)})}. 
$$
The projection of $\sZ(\mathbs, g)$ onto $L_{\mathrm{res}}^K$ is given by 
$$
\sZ(\mathbs, g)_{\mathrm{res}} = \langle \sZ(\mathbs, \cdot), \phi_{\mathrm{res}} \rangle \phi_{\mathrm{res}}(g) = \frac{1}{\mathrm{vol} \, (G(\bQ)\backslash G(\bA))} \frac{\Lambda(s-1)\Lambda(w)}{\Lambda(s+w)}. 
$$

\

Next, we take an orthonormal basis $\{ \phi_i \}_i$ for $L_{\mathrm{cusp}}^K$. We have 
$$
\sZ(s, g)_{\mathrm{cusp}} = \sum_i \langle \sZ(s, .), \phi_i \rangle \phi_i(g). 
$$
By Lemma~\ref{lemm: cuspidalpart}, $\sZ(s, e)_{\mathrm{cusp}}$ is holomorphic for $\Re s > 3/2+\delta$ and $\Re (s +w)  > 2$.

\

Finally we consider the projection onto the continuous spectrum. We have 
$$
\sZ(s, g)_{\mathrm{eis}} = \frac{1}{4 \pi} \int_{\bR} \langle \sZ(s, \cdot), E(it, \cdot) \rangle E(it, g) \, \mathrm dt.  
$$
The discussion in Section~\ref{subsec: constant} that 
$$
\sZ(s, g)_{\mathrm{eis}} 
=  - \frac{\Lambda(s-1)\Lambda(w)}{\Lambda(s+ w)} \, {\rm Res}_{y = 1/2} E(y, g) +   \frac{\Lambda(s+w-2)}{\Lambda(s+ w)} E(s-3/2, g) 
$$
$$
-  \frac{\Lambda(s+w-1)}{\Lambda(s+ w)} E(s-1/2, g)
 + f(s, w, g)
$$
with $f(s, w, g)$ a function which is holomorphic in the domain $\Re(s) > 3/2$ and $\Re (s+w) > 2$. 

\

We then have 
$$
\sZ(s) = \sZ(s, e)_{\mathrm{res}} + \sZ(s, e)_{\mathrm{cusp}}+ \sZ(s, e)_{\mathrm{eis}}
$$
$$
=   \frac{\Lambda(s-1)\Lambda(w)}{\Lambda(s+ w)} 
\left\{ \frac{1}{\mathrm{vol} \, (G(\bQ)\backslash G(\bA))} - {\rm Res}_{y = 1/2} E(y, e) \right\}
+   \frac{\Lambda(s+w-2)}{\Lambda(s+ w)} E(s-3/2, e) 
$$
$$
-   \frac{\Lambda(s+w-1)}{\Lambda(s+ w)} E(s-1/2, e)
 + \Phi(s, w)
$$
with $\Phi(s, w)$ a function holomorphic for $\Re s > 3/2+\delta$ and $\Re(s+w)> 2$. 

\

Consequently, we have proved the following statement: 

\begin{equation}\label{formula:main}
\sZ(s)
=   
\frac{\Lambda(s+w-2)}{\Lambda(s+ w)} E(s-3/2, e) 
-   \frac{\Lambda(s+w-1)}{\Lambda(s+ w)} E(s-1/2, e)
 + \Phi(s, w)
\end{equation}
with $\Phi(s, w)$ a function holomorphic for $\Re s > 3/2+\delta$ and $\Re(s+w) > 2$. This finishes the proof of Theorem \ref{mainthm3}. 

\subsection{The proof of Theorem~\ref{mainthm2}}

We now proceed to prove Theorem \ref{mainthm2} without the determination of the constant $C$. We restrict the function $\sZ(s, w)$ to the line $s = 2w$, and determine the order of pole and the leading term at $w=1$. We have 
$$
\sZ((2, 1)w) =   
 \frac{\Lambda(3w-2)}{\Lambda(3w)} E(2w-3/2, e) 
-   \frac{\Lambda(3w-1)}{\Lambda(3w)} E(2w-1/2, e)
 + \Phi(2w, w).
$$
The function 
$$
-\frac{\Lambda(3w-1)}{\Lambda(3w)} E(2w-1/2, e) + \Phi(2w , w)
$$
is holomorphic in the domain $\Re w > 3/4+ \delta/2$.  The function  
$$
h(w) :=   \frac{\Lambda(3w-2)}{\Lambda(3w)} E(2w-3/2, e)
$$
has a pole of order $2$ at $w=1$. The coefficient of $(w-1)^{-2}$ in the Taylor expansion of $h(w)$ is given by 
$$
\lim_{w \to 1} (w-1)^2 h(w) = \frac{1}{\Lambda(3)} \lim_{w \to 1} \frac{w-1}{3w-3}.\frac{(w-1) {\rm Res}_{s=1/2} E(s, e)}{2w-2} 
$$
$$
=  \frac{{\rm Res}_{s=1/2} E(s, e)}{6 \Lambda(3)}.
$$
Hence 
$$
\lim_{w \to 1} (w-1)^2 h(w) =  \frac{{\rm Res}_{s=1/2} E(s, e)}{6 \Lambda(3)} = \frac{1}{\zeta(3)}. 
$$
The asymptotic formula, modulo the determination of the constant $C$, now follows from Theorem A.1 of \cite{tauberian}.  A standard computation, as presented in e.g. the proof of Theorem 1 of \cite{abelian}, shows that 
$$
Z((2, 1)w) =  w \left(\frac{1/\zeta(3)}{(w-1)^2} + \frac{C}{w-1} \right) + g(w)
$$
$$
= \frac{1/\zeta(3)}{(w-1)^2} + \frac{C+ 1/\zeta(3)}{w-1} +  C + g(w), 
$$
with $g(w)$ holomorphic for $\Re w > 1-\eta$.  As a result, 
$$
\frac{\Lambda(3w-2)}{\Lambda(3w)} E(2w-3/2, e) = \frac{1/\zeta(3)}{(w-1)^2} + \frac{C+ 1/\zeta(3)}{w-1} +  \tilde{g}(w) 
$$
for $\tilde{g}(w)$ holomorphic in an open half plane containing $w=1$. So, in order to determine $C$ we need to determine the residue of the function appearing on the left hand side of this equation. 

\

A straightforward computation shows 
$$
\frac{\Lambda(3w-2)}{\Lambda(3w)} = \frac{1}{3 \Lambda(3)} \frac{1}{w-1} 
+ \frac{1}{\Lambda(3)} \left(\gamma - \frac{1}{2} \log \pi +  \frac{1}{2} \pi^{-1/2} \Gamma'\left(\frac{1}{2}\right) - \frac{\Lambda'(3)}{\Lambda(3)}\right) 
+ (w-1) \vartheta_1(w) 
$$
with $\vartheta_1(w)$ holomorphic in a neighborhood of $w=1$.  Here $\gamma$ is the Euler constant. 

\

We then recall the Kronecker Limit Formula, Theorem 10.4.6 of \cite{Cohen} and also Chapter 1 of \cite{Siegel}, in the following form. We have 
$$
E(s, g) = \frac{\pi}{2 \zeta(2s+ 1)} \left(\frac{1}{s-1/2} + C(g) + (s-1/2) \vartheta_2 (s)  \right)
$$
with $\vartheta_2(s)$ a function holomorphic in a neighborhood of $s=1/2$.  
The function $C(g)$ is described as follows. Since $\bQ$ has class number $1$, we have 
$$
\mathrm{GL}_2(\bA) = \mathrm{GL}_2(\bQ) \mathrm{GL}_2(\bR)^+ \mathrm{GL}_2(\hat\bZ). 
$$
Given $g \in \PGL_2(\bA)$, we choose a representative $g' \in \mathrm{GL}_2(\bA)$, 
we write $g' = \gamma g_\infty k$ with $\gamma \in \mathrm{GL}_2(\bQ), g_\infty \in \mathrm{GL}_2(\bR)^+, k \in \mathrm{GL}_2(\hat\bZ)$. 
We then set $\tau(g) = g_\infty\cdot i$ where the action is given by 
$$
\begin{pmatrix} \alpha & \beta \\ \gamma & \delta \end{pmatrix}\cdot z = \frac{\alpha z + \beta}{\gamma z + \delta}. 
$$ 
Note that since $g_\infty \in GL_2(\bR)^+$, $\Im (g_\infty\cdot i) >0$. We let $y(g) = \Im \tau(g)$. We then have 
$$
C(g) = 2 \gamma - 2 \log 2 - \log y(g) - 4 \log |\eta(\tau(g))|. 
$$
Multiplying out, we get 
$$
E(s, g) = \frac{3}{\pi} \frac{1}{s - 1/2} + \left(\frac{3}{\pi} C(g) - \frac{36 \zeta'(2)}{\pi^3} \right)+ (s-1/2) \vartheta_3(s)
$$
with $\vartheta_3(s)$ a function holomorphic in a neighborhood of $s=1/2$.  Hence, 
$$
E(2w - 3/2) = \frac{3}{2 \pi (w-1)} + \left(\frac{3}{\pi} C(e) - \frac{36 \zeta'(2)}{\pi^3} \right)+ (w-1) \vartheta_4(s)
$$
with $\vartheta_4(w)$ holomorphic in a neighborhood of $w=1$.  We also have 
$$
C(e) = 2 \gamma - 2 \log 2 - 4 \log | \eta(i)|. 
$$
Finally, 
$$
\frac{\Lambda(3w-2)}{\Lambda(3w)} E(2w-3/2, e) = 
$$
$$
\left\{\frac{1}{3 \Lambda(3)} \frac{1}{w-1} + \frac{1}{\Lambda(3)} 
\left(\gamma - \frac{1}{2} \log \pi +  \frac{1}{2} \pi^{-1/2} \Gamma'\left(\frac{1}{2}\right) - \frac{\Lambda'(3)}{\Lambda(3)}\right) + (w-1) \vartheta_1(w)\right\} 
$$
$$
\times \left\{\frac{3}{2 \pi (w-1)} + \left(\frac{3}{\pi} C(e) - \frac{36 \zeta'(2)}{\pi^3} \right)+ (w-1) \vartheta_4(s)\right\}
$$
$$
= \frac{1}{\zeta(3)} \frac{1}{(w-1)^2} + \left(\frac{3}{\pi} C(e) - \frac{36 \zeta'(2)}{\pi^3} \right)\frac{1}{3 \Lambda(3)} \frac{1}{w-1}
$$
$$
+ \frac{1}{\Lambda(3)} \left(\gamma - \frac{1}{2} \log \pi +  \frac{1}{2} \pi^{-1/2} \Gamma'\left(\frac{1}{2}\right) - \frac{\Lambda'(3)}{\Lambda(3)}\right)\frac{3}{2 \pi (w-1)}
 + \vartheta_5(w) 
$$
with $\vartheta_5(w)$ holomorphic in a neighborhood of $w=1$. Simplifying 
$$
\frac{\Lambda(3w-2)}{\Lambda(3w)} E(2w-3/2, e) = \frac{1}{\zeta(3)} \frac{1}{(w-1)^2} + \frac{A}{w-1} + \vartheta_5(w) 
$$
with 
$$
\zeta(3) A= 5 \gamma - 4 \log 2 - \frac{1}{2} \log \pi - \log |\eta(i)| + \frac{1}{2} \pi^{-1/2} \Gamma'(1/2) - \frac{24}{\pi^2} \zeta'(2) - \frac{\Lambda'(3)}{\Lambda(3)}. 
$$
Elementary computations show 
$$
\frac{1}{2\sqrt{\pi}} \Gamma'\left(\frac{1}{2}\right) = - \frac{1}{2} \gamma - \log 2 
$$
$$
\frac{\Lambda'(3)}{\Lambda(3)}  = -\frac{1}{2} \log \pi - \frac{1}{2} \gamma - \log 2 + 2 + \frac{\zeta'(3)}{\zeta (3)}. 
$$
Also we have the well-known identity 
$$
\eta(i) = \frac{\Gamma(1/4)}{2 \pi^{3/4}}.
$$
Putting everything together we conclude the following lemma:
\begin{lemm}
We have
$$
\frac{\Lambda(3w-2)}{\Lambda(3w)} E(2w-3/2, e) = \frac{1}{\zeta(3)} \frac{1}{(w-1)^2} + \frac{A}{w-1} + \vartheta_5(w) 
$$
where $\vartheta_5(w)$ is a holomorphic function in the domain $\Re(w) >1 - \eta$ for some $\eta >0$ and the constant $A$ is given by
$$
\zeta(3) A = 5 \gamma - 3 \log 2 + \frac{3}{4} \log \pi - \log \Gamma\left(\frac{1}{4}\right) - \frac{24}{\pi^2} \zeta'(2) - \frac{\zeta'(3)}{\zeta(3)} - 3. 
$$
\end{lemm}

\section{Manin's conjecture with Peyre's constant}\label{sect:peyre} 

\subsection{The anticanonical class}

In this section, we verify that the anticanonical class satisfies Manin's conjecture with Peyre's constant, 
hence finishing the proof of Theorem \ref{mainthm} for the anticanonical class. For the definition of Peyre's constant, see \cite{peyre95}. 
The N\'eron-Severi lattice $\mathrm{NS}(X)$ is generated by $E$ and $H$ where $H$ is the pullback of the hyperplane class on $\bP^3$. 
Let $N_1(X)$ is the dual lattice of $\mathrm{NS}(X)$ which is generated by the dual basis $E^*$ and $H^*$. 
The cone of effective divisors $\Lambda_{\mathrm{eff}}(X) \subset \mathrm{NS}(X)_{\bR}$ is generated by $E$ and $H-E$. 
We denote the dual cone of the cone of effective divisors by $\mathrm{Nef}_1(X)$.  Any element of this dual cone is called a nef class. 
Let $\mathrm da$ be the normalized haar measure on $N_1(X)_{\bR}$ such that $\mathrm{vol}(N_1(X)_{\bR} / N_1(X)) = 1$. 
The alpha invariant (see \cite[Definition 4.12.2]{Tsc09}) is given by
$$ 
\alpha(X) = \int_{\mathrm{Nef}(X)} e^{-\langle-K_X, a\rangle} \, \mathrm d a
= \int_{\{x \geq 0, y-x \geq 0\}}e^{-(4y-x)} \, \mathrm dx \, \mathrm dy = \frac{1}{12}.
$$

Next we compute the Tamagawa number.
Let $\omega \in \Gamma(\PGL_2, \Omega^3_{\PGL_2/\mathrm{Spec}(\bZ)})$ be a nonzero relative top degree invariant form over $\mathrm{Spec}(\bZ)$. This is unique up to sign.
Our height induces a natural metrization on $\mathcal O(K_X)$, and it follows from the construction that
$$
\mathsf H_v(g_v, s=2, w=1)^{-1} = \| \omega \|_v.
$$
The normalized haar measure $\mathrm dg_v$ is given by $|\omega|_v/a_v$ at non-archimedean places, hence we have
\begin{align*}
\tau_{X, v}(X(F_v)) &= \int_{X(F_v)} \,\mathrm d\tau_{X, v} 
=  \int_{X(F_v)} \,\mathrm d\frac{|\omega|_v}{\| \omega \|_v} \\&= a_v \int_{G(F_v)} \sH_v(g_v, s=2, w=1) ^{-1} \, \mathrm dg_v =
(1-q_v^{-2}) \frac{1-q_v^{-3}}{(1-q_v^{-1}) (1-q_v^{-1})}.
\end{align*}
For the infinite place $v = \infty$, we have
$$
\mathsf H_\infty(g_\infty, s=2, w=1)^{-1} = \| \omega \|_\infty, \quad \mathrm dg_\infty = \frac{|\omega|_\infty}{\pi}.
$$
Hence we conclude that
$$
\tau_{X, \infty}(X(\bR)) = \pi \int_{X(\bR)} \sH^{-1} (g, s = 2, w = 1) \, \mathrm dg 
= \pi \cdot \sqrt{\pi} \frac{\Gamma(1/2)^2}{\Gamma(3/2)} = 2\pi^2.
$$
For any non-archimedean place $v$, the local $L$-function at $v$ is given by
$$
L_v(s, \mathrm{Pic}(\mathcal X_{\bar{k}_v})_{\bQ}) := \det(1-q_v^{-s} \mathrm{Fr}_v \mid \mathrm{Pic}(\mathcal X_{\bar{k}_v})_{\bQ} )^{-1} = (1-q_v^{-s})^{-2}.
$$
We define the global $L$-function by
$$
L(s, \mathrm{Pic} (X_{\bar{\bQ}}) ):= \prod_{ v\in \mathrm{Val}(\bQ)_{\mathrm{fin}}} L_v (s, \mathrm{Pic}(\mathcal X_{\bar{k}_v})_{\bQ} ) = \zeta_F(s)^2.
$$

Let $\lambda_v = 1/L_v(1, \mathrm{Pic}(\mathcal X_{\bar{k}_v})_{\bQ}) = (1-q_v^{-1})^2$. Then we have
$$
\tau(-\mathcal K_X) = \zeta_{F *}(1)^2 \prod_v \lambda_v \tau_{X, v}(X(F_v)) = 2\pi^2 \prod_p (1-p^{-2})(1-p^{-3}) = 2\pi^2 \frac{1}{\zeta(2) \zeta(3)} = \frac{12}{\ \zeta(3)}.
$$
Thus the leading constant is
$$
c(-\mathcal K_X) = \alpha (X) \tau(-\mathcal K_X)= \frac{1}{\zeta(3)}.
$$

\subsection{Other big line bundles: the rigid case}
Consider the following $\bQ$-divisor:
\[
 L = x\tilde{D}+yE,
\]
where $x,y$ are rational numbers.
The divisor $L$ is big if and only if $x > 0$ and $x+y >0$.
We define the following invariant:
\begin{align*}
 & a(L) = \inf \{ t \in \bR \mid tL + K_X \in \Lambda_{\mathrm{eff}}(X) \}, \\
 & b(L) = \text{the codimension of the minimal face containing $a(L)[L]+[K_X]$ of $\Lambda_{\mathrm{eff}}(X)$}.
\end{align*}
It follows from Theorem~\ref{mainthm3} that the height zeta function $\mathsf Z(sL)$ has a pole at $s = a(L)$ of order $b(L)$.
Thus, to verify Manin's conjecture, the only issue is the leading constant i.e., the residue of $\mathsf Z(sL)$.
Tamagawa numbers for general big line bundles are introduced by Batyrev and Tschinkel in \cite{bt-tamagawa}.
(See \cite[Section 4.14]{Tsc09} as well.)
The definition is quite different depending on whether the adjoint divisor $a(L)L+K_X$ is rigid or not.
Here we assume that $a(L)L+K_X$ is a non-zero rigid effective $\bQ$-divisor.

In this case, the adjoint divisor $a(L)L+K_X$ is proportional to $E$.
This happens if and only if $2y-x >0$.
In this situation, we have $a(L)= 2/x$ and $\sZ(sL)$ has a pole at $s = a(L)$ of order one.
By Theorem~\ref{mainthm3}, we have
\[
 \lim_{s \rightarrow a(L)}(s-a(L))\sZ(sL) = \frac{\Lambda(2y/x)}{\Lambda(2+ 2y/x)}\frac{3}{\pi x}.
\]
Recall that 
\[
 \frac{\Lambda(s-1)\Lambda(w)}{\Lambda(s+w)} = \int_{G(\bA)} \sH(g, s, w)^{-1} \, \mathrm dg
 = \frac{\pi}{6}\int_{G(\bA)} \sH(g, s-2, w-1)^{-1} \, \mathrm d\tau_G,
\]
where $\tau_G$ is the Tamagawa measure on $G(\bA)$ defined by $\tau_G = \prod_v \frac{|\omega|_v}{\|\omega\|_v}$.
We denote this function by $\widehat{\sH}(s,w)$.
It follows from the computation of \cite[Section 4.4]{volume} that 
\[
 \lim_{s \rightarrow a(L)} (s-a(L))\widehat{\sH}(sL) = \frac{1}{x}\frac{\Lambda(2y/x)}{\Lambda(2 + 2y/x)}
 = \frac{\pi}{6x} \int_{X^\circ(\bA)}\sH(x, a(L)L + K_X)^{-1} \, \mathrm d \tau_{X^\circ},
\]
where $X^\circ = X \setminus E$, and $\tau_{X^\circ}$ is the Tamagawa measure on $X^\circ$.
Thus the Tamagawa number $\tau(X, \mathcal L)$ in the sense of \cite{bt-tamagawa} is given by
\[
 \tau(X, \mathcal L) = \int_{X^\circ(\bA)}\sH(x, a(L)L + K_X)^{-1} \, \mathrm d \tau_{X^\circ} = \frac{6}{\pi}\frac{\Lambda(2y/x)}{\Lambda(2 + 2y/x)}.
\]
On the other hand, the alpha invariant is given by
\[
 \alpha(X, L) = \frac{1}{2x}.
\]
Thus we conclude
\[
 c(X, \mathcal L) = \frac{\Lambda(2y/x)}{\Lambda(2+ 2y/x)}\frac{3}{\pi x}.
\]

\subsection{The non-rigid case}
Again, we consider a big $\bQ$-divisor $L = x\tilde{D} + yE$,
and assume that $a(L)L+K_X$ is not rigid, i.e., some multiple of $a(L)L+K_X$ defines the Iitaka fibration.
This happens exactly when $2y - x<0$ and $a(L)$ is given by $3/(x+y)$.
The adjoint divisor $a(L)L+K_X$ is proportional to $\tilde{D}-E$ which is semiample
and defines a morphism $f : X \rightarrow P\backslash G = \bP^1$.

It follows from Theorem~\ref{mainthm3} that
\[
 \lim_{s \rightarrow a(L)}(s - a(L))\sZ(sL) = \frac{1}{(x+y)\Lambda(3)}E\left(\frac{3x}{x+y} - \frac{3}{2}, e\right).
\]
Again it follows from the computation of \cite[Section 4.4]{volume} that
\[
 \lim_{s \rightarrow 1}(s-1)^2 \widehat{\sH}(-sK_X) = \frac{1}{2\Lambda(3)} = \frac{\pi}{12}\int_{X(\bA)} \, \mathrm d\tau_X,
\]
where $\tau_X$ is the Tamagawa measure on $X$.

On the other hand, it follows from the definition of the Tamagawa measure that
\begin{align*}
 \int_{X(\bA)} \, \mathrm d\tau_X 
 &= \prod_{v : \mathrm{fin}} (1-p_v^{-1})^2 \int_{X(\bQ_v)} \mathrm d \tau_{X, v} \cdot \int_{X(\bR)} \mathrm d \tau_{X, \infty} \\
 &= \frac{6}{\pi}\prod_{v : \mathrm{fin}} (1-p_v^{-1})^2 \int_{G(\bQ_v)} \sH(g_v, s=2, w=1)^{-1}\, \mathrm d g_v \cdot 
 \int_{G(\bR)} \sH(g_\infty, s=2, w=1)^{-1} \, \mathrm d g_\infty \\
 &= \frac{6}{\pi}\prod_{v : \mathrm{fin}} (1-p_v^{-1})^2 \int_{P(\bQ_v)} \sH(h_v, s=2, w=1)^{-1}\, \mathrm d_l h_v \cdot 
 \int_{P(\bR)} \sH(h_\infty, s=2, w=1)^{-1} \, \mathrm d_l h_\infty,
\end{align*}
where $P$ is the standard Borel subgroup and $\mathrm d_l h_v$ is the left invariant haar measure on $P(\bQ_v)$ given by 
$\mathrm d_l h_v = |a|_v^{-1} \mathrm d a_v^\times \mathrm dx_v$.
Let $\sigma$ be a top degree left invariant form on $P$. We have
\[
 \mathrm{div} (\sigma) = -2\tilde{D}|_S -  E|_S,
\]
where $S$ is the Zariski closure of $P$ in $X$. Hence
\begin{align*}
 \int_{X(\bA)} \, \mathrm d\tau_X 
 &= \frac{6}{\pi}\prod_{v : \mathrm{fin}} (1-p_v^{-1}) \int_{P(\bQ_v)} \sH(h_v, s=2, w=1)^{-1}\, \mathrm d |\sigma|_v \cdot 
 \int_{P(\bR)} \sH(h_\infty, s=2, w=1)^{-1} \, \mathrm d |\sigma|_\infty \\
 &= \frac{6}{\pi}\prod_{v : \mathrm{fin}} (1-p_v^{-1}) \int_{S(\bQ_v)} \mathrm d \tau_{S, v} \cdot 
 \int_{S(\bR)}  \mathrm d \tau_{S, \infty} \\
 &= \frac{6}{\pi}\int_{S(\bA)} \mathrm d \tau_S.
\end{align*}

Recall that we have the following height function on $\bP^1$:
\[
 h_{\bP^1}(c:d)= \prod_{v : \mathrm{fin}} \max \{ |c|_v, |d|_v \} \cdot \sqrt{c^2+d^2} : \bP^1(\bQ) \rightarrow \bR_{>0},
\]
and our Eisenstein series is related to this height function by
\[
 E(s, e) = \sZ_{\bP^{1}}(2s+1) = \sum_{z \in \bP^1(\bQ)} h_{\bP^1}(z)^{-(2s+1)}.
\]
For each $z = (c:d) \in \bP^1$ where $c, d$ are coprime integers, choose integers $a, b$ so that $ad-bc = 1$
and let $g_z= \begin{pmatrix}a & b \\ c&d\end{pmatrix}$.
Then the fiber $f^{-1}(z) = S_z$ is the translation $S\cdot g_z$.
The alpha invariant is given by
\[
 \alpha(S_z, L) = \frac{1}{x+y}.
\]
Next the Tamagawa number is given by
\begin{align*}
 \int_{S_z(\bA)} \, \mathrm d\tau_{S_z} 
 &= \prod_{v : \mathrm{fin}} (1-p_v^{-1}) \int_{S_z(\bQ_v)} \mathrm d \tau_{S_z, v} \cdot 
 \int_{S_z(\bR)}  \mathrm d \tau_{S_z, \infty} \\
 &=\prod_{v : \mathrm{fin}} (1-p_v^{-1}) \int_{P(\bQ_v)} \sH(h_v g_z, s=2, w=1)^{-1}\, \mathrm d |\sigma|_v \cdot 
 \int_{P(\bR)} \sH(h_\infty g_z, s=2, w=1)^{-1} \, \mathrm d |\sigma|_\infty \\
 &= \prod_{v : \mathrm{fin}} (1-p_v^{-1}) \int_{P(\bQ_v)} \sH(h_v, s=2, w=1)^{-1}\, \mathrm d |\sigma|_v \cdot 
 \int_{P(\bR)} \sH(h_\infty g_z, s=2, w=1)^{-1} \, \mathrm d |\sigma|_\infty.
\end{align*}
Note that $g_z \in G(\bZ)$. Then we have
\begin{align*}
 \int_{P(\bR)} \sH(h_\infty g_z, s=2, w=1)^{-1} \, \mathrm d |\sigma|_\infty
 &= \int_{\bR} \int_{\bR^\times} \sH\left(\begin{pmatrix} t& x \\ & 1\end{pmatrix}\begin{pmatrix} a& b \\ c& d\end{pmatrix}, s=2,w=1\right)^{-1}
 |t|^{-1} \mathrm dt^\times \mathrm dx \\
 &= \int_{\bR} \int_{\bR^\times} \sH\left(\begin{pmatrix} t& x \\ & 1\end{pmatrix}\begin{pmatrix} \frac{1}{c^2+d^2}& \frac{ac+bd}{c^2+d^2} \\ & 1\end{pmatrix}, s=2,w=1\right)^{-1}
 |t|^{-1} \mathrm dt^\times \mathrm dx \\
 &= (c^2+d^2)^{-1}\int_{\bR} \int_{\bR^\times} \sH\left(\begin{pmatrix} t& x \\ & 1\end{pmatrix}, s=2,w=1\right)^{-1}
 |t|^{-1} \mathrm dt^\times \mathrm dx \\
 &= h_{\bP^1}(z)^{-2} \int_{S(\bR)} \mathrm d \tau_{S, \infty}.
\end{align*}
Thus we have
\[
 \int_{S_z(\bA)} \, \mathrm d\tau_{S_z} = h_{\bP^1}(z)^{-2} \int_{S(\bA)} \mathrm d \tau_{S}.
\]
Let $H$ be the pull back of the ample generator via $f : X \rightarrow \bP^1$.
Then we have $a(L)L+K_X \sim 2\frac{x-2y}{x+y}H$.
We conclude that
\begin{align*}
 c(X, \mathcal L) &= \sum_{z \in \bP^1(\bQ)} \alpha(S_z, L) \sH(g_z, a(L)L+K_X)^{-1} \tau(S_z, \mathcal L) \\
 &= \frac{1}{x+y} \sum_{z \in \bP^1(\bQ)} h_{\bP^1}(z)^{-2\frac{x-2y}{x+y}} h_{\bP^1}(z)^{-2} \tau(S, \mathcal L)\\
 &= \frac{1}{x+y} \frac{1}{\Lambda(3)}  \sum_{z \in \bP^1(\bQ)}h_{\bP^1}(z)^{-\frac{4x-2y}{x+y}}\\
 &= \frac{1}{(x+y)\Lambda(3)} E\left(\frac{3x}{x+y}-\frac{3}{2}, e\right).
\end{align*}
We have verified Manin's conjecture in this case.

\section{Appendix: Special functions}
\label{sec: appendix}

Here we collect some important facts about special functions which we use in our analysis of the height zeta function. First we state the Stirling formula for the Gamma function:
\begin{theo}{\cite[p.220, B.8]{Iwa95}}
\label{theo: stirling}
Let $\sigma$ be a real number such that $|\sigma| \leq 2$. Then we have
\[
\Gamma(\sigma + it) = (2\pi)^{\frac{1}{2}} t^{\sigma - \frac{1}{2}} e^{-\frac{\pi t}{2}}
\left(\frac{t}{e}\right)^{it}(1+O(t^{-1}))
\]
for $t >0$.
\end{theo}

Next we list some estimates for the modified Bessel function of the second kind:
\begin{prop}{\cite[Proposition 3.5]{har03}}
For any real number $\sigma >0$ and $\epsilon>0$, the following uniform estimate holds in the vertical stripe defined by $|\Re(\mu)| \leq \sigma$:
\[
e^{\pi|\Im(\mu)|/2} K_\mu(x) \ll
\begin{cases}
(1 + |\Im(\mu)|)^{\sigma + \epsilon} x^{-\sigma - \epsilon}, & 0 < x \leq 1 + \pi |\Im(\mu)|/2; \\
e^{-x +\pi|\Im(\mu)|/2 } x^{-1/2}, & 1 + \pi |\Im(\mu)|/2< x.
\end{cases}
\]
Here the implied constant depends on $\sigma$ and $\epsilon$.
\end{prop}

\begin{lemm}{\cite[p. 905, (3.16)]{Si85}}
Suppose that $|\Re(\mu)| < \frac{1}{2}$. Then we have
\[
K_\mu(2\pi) = \frac{\Gamma(\mu)}{2\pi^\mu}(1+O(|\Im(\mu)|^{-1}))
\]
when $|\Im(\mu)|$ is sufficiently large.
\end{lemm}
%\begin{proof}
%Since we have $K_\mu(x) = K_{-\mu}(x)$, we may assume that the real part of $\mu$ is non-positive.
%Next we have the following recurrence formula:
%\[
 %K_\mu(x) = K_{\mu-2}(x) + \frac{2(\mu-1)}{x}K_{\mu-1}(x).
%\]
%Suppose that $\nu$ is a complex number such that $-5/2 < \Re(\nu)\leq -1$.
%Let $I_\nu(x)$ be the modified Bessel function of the first kind.
%According to \cite[p. 225, B.32]{Iwa95}, we have the following asymptotic formula
%\[
 %I_\nu(x) \sim \frac{(x/2)^\nu}{\Gamma(\nu + 1)},
%\]
%as $|\Im(\nu)| \rightarrow \infty$. Hence we obtain
%\[
 %\frac{I_\nu(x)}{I_{-\nu}(x)} \sim \frac{(x/2)^{2\nu}\Gamma(1-\nu)}{\Gamma(1+\nu)}  
 %=\frac{\pi}{\sin (\pi \nu)}\frac{(x/2)^{2\nu}}{\nu\Gamma(\nu)^2} = \frac{\pi}{\sin (\pi \nu)}(x/2)^{2\nu}\nu\Gamma(-\nu)^2.
%\]
%Here we use two functional equations of the Gamma function:
%\[
 %\Gamma(\nu)\Gamma(1-\nu) = \frac{\pi}{\sin (\pi \nu)}, \quad \Gamma(\nu + 1) = \nu\Gamma(\nu).
%\]
%Using the Stirling formula (Theorem~\ref{theo: stirling}), we have
%\[
 %\left|\frac{I_\nu(x)}{I_{-\nu}(x)}\right| \sim \left(\frac{x}{2|\Im(\nu)| }\right)^{2\Re(\nu)},
%\]
%as $|\Im(\nu)| \rightarrow \infty$. Since $K_\nu(x)$ is given by
%\[
 %K_\nu(x) = \frac{\pi}{2}\frac{I_{-\nu}(x) - I_{\nu}(x)}{\sin (\pi \nu)},
%\]
%we conclude that
%\[
 %|K_\nu(x)| \sim \frac{\pi}{2\sqrt{2\pi}}\left(\frac{x}{2|\Im(\nu)| }\right)^{\Re(\nu)} \frac{e^{-\pi |\Im(\nu)|/2}}{\sqrt{|\Im(\nu)|}},
%\]
%as $|\Im(\nu)| \rightarrow \infty$.
%Now using the recurrence relation, we obtain
%\[
 %|K_\mu(2\pi)| \gg e^{-\pi |\Im(\nu)|/2} |\Im(\nu)|^{3/2 - \Re(\mu)}
%\]
%\end{proof}

\begin{theo}[\cite{Ko58} and \cite{Vi58}]
\label{theo: zeta}
 As $|t| \rightarrow \infty$, we have
 \[
  |\zeta(1+it)| \gg (\log |t|)^{-1}.
 \]

\end{theo}

\bibliographystyle{alpha}
\bibliography{PGL2}

\end{document}